\documentclass{amsart}

\usepackage{hyperref}
\usepackage[small,nohug,heads=vee]{diagrams}
\diagramstyle[labelstyle=\scriptstyle]
\usepackage{amssymb}
\usepackage{amsmath}
\usepackage{amsthm}
\usepackage{enumerate}
\usepackage{graphicx}

\theoremstyle{plain}
\newtheorem{theorem}{Theorem}[section]
\newtheorem{proposition}[theorem]{Proposition}
\newtheorem{lemma}[theorem]{Lemma}
\newtheorem{corollary}[theorem]{Corollary}

\newtheorem{conjecture}[theorem]{Conjecture}

\theoremstyle{definition}

\newtheorem{example}{Example}[subsection]

\theoremstyle{remark}
\newtheorem*{remark}{Remark}

\newtheorem*{acknowledgment}{Acknowledgment}

\renewcommand {\tilde} {\widetilde}
\renewcommand{\Re}{\mathrm{Re\,}}

\begin{document}

\title[Computation of analytic capacity]{Computation of analytic capacity and applications to the subadditivity problem }

\date{February 19, 2013}

\author[M. Younsi]{Malik Younsi}
\thanks{First author supported by the Vanier Canada Graduate Scholarships program.}
\address{D\'epartement de math\'ematiques et de statistique, Pavillon Alexandre--Vachon, $1045$ av. de la M\'edecine, Universit\'e Laval, Qu\'ebec (Qu\'ebec), Canada, G1V 0A6.}
\email{malik.younsi.1@ulaval.ca}

\author[T. Ransford]{Thomas Ransford}
\thanks{Second author supported by grants from NSERC and the Canada Research Chairs program.}
\address{D\'epartement de math\'ematiques et de statistique, Pavillon Alexandre--Vachon, $1045$ av. de la M\'edecine, Universit\'e Laval, Qu\'ebec (Qu\'ebec), Canada, G1V 0A6.}
\email{ransford@mat.ulaval.ca}


\begin{abstract}
We develop a least-squares method for computing the analytic capacity of compact plane sets with piecewise-analytic boundary. The method furnishes rigorous upper and lower bounds which converge to the true value of the capacity. Several illustrative examples are presented. We are led to formulate a conjecture which, if true, would imply that analytic capacity is subadditive. The conjecture is proved in a special case.
\end{abstract}

\maketitle

\section{Introduction}
\label{sec1}
Let $K$ be a compact subset of $\mathbb{C}$ and let $\Omega$ be the complement of $K$ in the Riemann sphere, i.e. $\Omega:= \mathbb{C}_{\infty}
\setminus K$. The \textit{analytic capacity} of $K$ is
$$
\gamma(K):= \sup\{|f'(\infty)|: f \in H^{\infty}(\Omega), \|f\|_{\infty} \leq 1\}.
$$
Here $f'(\infty):=\lim_{z \rightarrow \infty} z(f(z)-f(\infty))$ denotes the coefficient of $z^{-1}$ in the Laurent expansion of $f(z)$ near infinity, and $H^{\infty}(\Omega)$ denotes the class of all bounded holomorphic functions in $\Omega$.

Analytic capacity of compact sets was first introduced by Ahlfors \textbf{\cite{AHL}}, in order to study Painlev\'e's problem of finding a geometric characterization of the compact sets $K$ that have the property that every bounded holomorphic function in $\Omega$ is constant. These compact sets are called \textit{removable} and are precisely those of zero analytic capacity. Painlev\'e's problem has been extensively studied in the last decades and is now considered solved. See e.g. \textbf{\cite{TOL2}} for a survey of Painlev\'e's problem and related results.

The study of analytic capacity became even more interesting with Vitushkin's work on uniform rational approximation \textbf{\cite{VIT}}. Vitushkin showed that analytic capacity plays a central role in the theory of uniform rational approximation of holomorphic functions on compact subsets of the plane. See e.g. \textbf{\cite{ZALC}} for a survey of the applications of analytic capacity to this type of problem.

We also mention that analytic capacity is used in fluid dynamics to study the $2$-dimensional velocity fields induced by several obstacles, see e.g. \textbf{\cite{MUR}}.

In this article, we are primarily interested in the computation of analytic capacity. More precisely, our main objectives are
\begin{itemize}
\item To obtain a quick and efficient method to compute the analytic capacity of ``nice'' compact sets.
\item To use this method to investigate the (still open) \textit{subadditivity problem} for analytic capacity.
\end{itemize}

The article is structured as follows. Section \ref{sec2} contains the necessary preliminaries on analytic capacity. In Section \ref{sec3}, we obtain some estimates for the analytic capacity of a compact set with $C^{\infty}$ boundary. Then, in Section \ref{sec4}, we prove that the same estimates remain valid in the case of compact sets with piecewise-analytic boundary (subject to suitable modifications). The proof relies on properties of the \textit{Smirnov classes} $E^p(\Omega)$ on finitely connected domains. In Section \ref{sec5}, we use these estimates to obtain a numerical method for the computation of analytic capacity. The method gives upper and lower bounds for the analytic capacity $\gamma(K)$, and we prove that these bounds can be made arbitrarily close, thus converging to the true value of $\gamma(K)$. In Section \ref{sec6}, we present several numerical examples, in the case of analytic boundary as well as in the case of piecewise-analytic boundary.

The last two sections, Section \ref{sec7} and  Section \ref{sec8}, are dedicated to the study of the subadditivity problem for analytic capacity. The problem is the following:
\emph{Is it true that $\gamma$ is subadditive, in the sense that
$$
\gamma(E \cup F) \leq \gamma(E) + \gamma(F)
$$
for all compact sets $E,F \subseteq \mathbb{C}$ ?}
Vitushkin conjectured in \textbf{\cite{VIT}} that analytic capacity is \textit{semi-additive}, i.e.
$$
\gamma(E \cup F) \leq C(\gamma(E) + \gamma(F))
$$
for some universal constant $C$. He gave various applications of this inequality to rational approximation. Semi-additivity of analytic capacity was proved by Tolsa \textbf{\cite{TOL}}. In fact, Tolsa proved that analytic capacity is countably semi-additive. However, it is still unknown whether or not one can take the constant $C$ equal to $1$. Davie \textbf{\cite{DAV}} gives some applications of the subadditivity of analytic capacity to rational approximation theory.

In Section \ref{sec7}, we first use a discrete approach to analytic capacity (introduced by Melnikov \textbf{\cite{MEL}}) to prove that the subadditivity of analytic capacity is equivalent to the subadditivity in the special case where the compact sets $E,F$ are disjoint finite unions of disjoint disks, all with the same radius. This result is quite convenient because the numerical method described in Section \ref{sec5} is very efficient for computing the analytic capacity of such compact sets. Then we use a discrete version of analytic capacity, also introduced by Melnikov \textbf{\cite{MEL}}, to obtain a result regarding the behavior of the ratio
$$
\frac{\gamma(E \cup F)}{\gamma(E) + \gamma(F)}
$$
as $r \rightarrow 0$, where $E$ and $F$ are disjoint finite unions of disjoint disks, all with radius~$r$.

Finally, in Section \ref{sec8}, we formulate a conjecture based, among other things, on numerical evidence. A proof of this conjecture would imply that analytic capacity is subadditive. We end the section by giving a proof in a special case.

In the article, we shall use the letter $\Omega$ to denote a domain in the Riemann sphere, that is, an open and connected subset of $\mathbb{C}_\infty$. Furthermore, when $\Omega$ is said to be a finitely connected domain with analytic (respectively piecewise-analytic) boundary, we mean that the boundary of $\Omega$ consists of a finite number of pairwise disjoint analytic (respectively piecewise-analytic) Jordan curves. Finally, we shall use $A(\Omega)$ to denote the set of complex-valued functions that are continuous in $\overline{\Omega}$, the closure of $\Omega$ in $\mathbb{C}_\infty$, and holomorphic in $\Omega$.

\section{Preliminaries on analytic capacity}
\label{sec2}

In a sense, analytic capacity measures the size of a set as a non-removable singularity for bounded holomorphic functions. A direct consequence of the definition is that $\gamma$ is \textit{monotone}:
$$
K_1 \subseteq K_2 \Rightarrow \gamma(K_1) \leq \gamma(K_2).
$$
It is also easy to prove that
$$
\gamma(aK +b) = |a|\gamma(K)
$$
for every $a,b \in \mathbb{C}$ and compact set $K$. In particular, $\gamma$ is invariant under translation.

Analytic capacity is also \textit{outer regular}, in the sense that, if
$$
K_1 \supseteq K_2 \supseteq K_3 \supseteq \dots
$$
is a decreasing sequence of compact sets, and if $K:= \cap_n K_n$, then
$\gamma(K_n) \rightarrow \gamma(K)$
as $n \rightarrow \infty$.

It is well known that, for every compact set $K$, there exists an extremal function $f$ for $\gamma(K)$, that is, a function $f$ holomorphic in $\Omega$ with $|f| \leq 1$ in $\Omega$ and $f'(\infty)=\gamma(K)$. In the case $\gamma(K)>0$, this function $f$ is unique in the unbounded component of $\Omega$ and is called the \textit{Ahlfors function} for $K$. One verifies easily that the Ahlfors function $f$ vanishes at $\infty$.

From Schwarz's lemma, it follows that, if $K$ is connected, then $f$ is the conformal map of $\Omega$ onto the unit disk $\mathbb{D}$ with $f(\infty)=0$ and $f'(\infty)>0$. As a consequence, we get that the analytic capacity of a closed disk equals the radius, and the analytic capacity of a closed line segment equals a quarter of the length. See e.g. \cite{GAR}.

\subsection{Finitely connected domains with analytic boundary}
Let $K$ be a compact set in the plane and again denote by $\Omega$ the complement of $K$ in $\mathbb{C}_{\infty}$. The following theorem says that, under certain assumptions on $K$, the Ahlfors function behaves nicely:

\begin{theorem}[Ahlfors \cite{AHL}]
\label{Theo1}
Let us assume that $\Omega$ is a finitely connected domain whose boundary consists of $n$ Jordan curves. In this case, the Ahlfors function $f$ is an $n$-to-$1$ branched covering of $\Omega$ onto the unit disk. Moreover,
\begin{enumerate}[\rm(i)]
\item $f$ extends continuously to $\partial \Omega$, the boundary of $\Omega$, so that $f \in A(\Omega)$;
\item $|f| \equiv 1$ on $\partial \Omega$;
\item $f$ maps each of the $n$ boundary curves homeomorphically onto the unit circle.
\end{enumerate}
\end{theorem}

By the Schwarz reflection principle, if in addition each boundary curve is analytic, then $f$ extends analytically across the boundary.

One way to prove Theorem \ref{Theo1} is to use the following result:

\begin{theorem}
\label{Theo2}
Suppose that $\Omega$ is a finitely connected domain with analytic boundary. Then there exists a holomorphic function $\psi$ in $\Omega$ which is the unique solution to the dual extremal problem
$$
\int_{\partial \Omega}|\psi(\zeta)||d\zeta|
= \inf \left\{\int_{\partial \Omega} |h(\zeta)||d\zeta|: h \in A(\Omega), h(\infty)=\frac{1}{2\pi i} \right\}.
$$
Moreover, $\psi$ has the following properties:
\begin{enumerate}[\rm(i)]
 \item $\psi \in A(\Omega)$ and extends analytically across $\partial \Omega$;
 \item $\psi(\infty) = 1/2 \pi i$;
 \item $\psi$ represents evaluation of the derivative at $\infty$, in the sense that, for all $g \in A(\Omega)$,
$$
g'(\infty) = \int_{\partial \Omega} g(\zeta)\psi(\zeta)d\zeta;
$$
 \item $\int_{\partial \Omega} |\psi(\zeta)| |d\zeta| = \gamma(K)$;
\item The extension of $\psi$ has an analytic logarithm. In particular, there exists a function $q \in A(\Omega)$ such that $q(\infty)=1$ and
$$
q(z)^2 = 2 \pi i \psi(z) \qquad (z \in \overline{\Omega}).
$$
\end{enumerate}
\end{theorem}

The above is essentially due to Garabedian \cite{GARA}. The function $\psi$ is usually called the \textit{Garabedian function} for $\Omega$. See also \textbf{\cite[Theorem 4.1]{GAR}}.

We end this subsection by remarking that the function $q$ in the above theorem is, up to a multiplicative constant, a reproducing kernel for the Hilbert space $H^2(\Omega)$.

Indeed, recall that, for any domain $\Omega$ and for $0<p<\infty$, the Hardy space $H^p(\Omega)$ is the class of all functions $h$ holomorphic in $\Omega$ such that the subharmonic function $|h|^p$ has a harmonic majorant. This definition is conformally invariant and coincides with the classical one when $\Omega$ is the unit disk. If $\Omega$ is a finitely connected domain with analytic boundary and if $\infty \in \Omega$, then $H^2(\Omega)$ is a Hilbert space, with respect to the scalar product
$$
\langle g,h \rangle = \int_{\partial \Omega} g(z)\overline{h(z)} |dz|,
$$
in which evaluation at $\infty$ is continuous. Hence there is a unique function $S(z,\infty) \in H^2(\Omega)$, called the \textit{Szeg\H o kernel function} for $\infty$, such that
$$
g(\infty) = \int_{\partial \Omega} g(z)\overline{S(z,\infty)}|dz| \qquad (g \in H^2(\Omega)).
$$

If $q$ is the function in $(v)$ of the above theorem, then we have
$$
q(z)= 2\pi \gamma(K) S(z,\infty)
\qquad (z \in \Omega).
$$
See e.g. \cite[Theorem 4.3]{GAR}.

\subsection{Transformation of the Ahlfors and Garabedian functions under conformal mapping}

The goal of this subsection is to describe how the Ahlfors and Garabedian functions transform under conformal mapping.

As before, let $K$ be a compact set in the plane, and set $\Omega:= \mathbb{C}_{\infty} \setminus K$. Suppose that $\Omega$ is a finitely connected domain whose boundary consists of a finite number of pairwise disjoint Jordan curves.

By repeated applications of the Riemann mapping theorem, there exists another compact set $\tilde{K}$ whose complement $\tilde{\Omega}$ is a finitely connected domain with analytic boundary and is conformally equivalent to $\Omega$. Denote by $F: \Omega \rightarrow \tilde{\Omega}$ the conformal map thereby obtained, normalized so that $F(\infty)=\infty$.

It is well known that every conformal map of a Jordan domain onto the unit disk extends to a homeomorphism of the closure of the domain onto the closed unit disk. Hence, by construction, $F$ extends to an homeomorphism of $\overline{\Omega}$ onto $\overline{\tilde{\Omega}}$. Write
$$
F(z)=a_1z+a_0+\frac{a_{-1}}{z}+\frac{a_{-2}}{z^2}+ \dots
$$
near infinity. The following proposition relates the Ahlfors functions for $K$ and $\tilde{K}$:

\begin{proposition}
\label{prop1}
Let $f, \tilde{f}$ be the Ahlfors functions for $K$ and $\tilde{K}$ respectively. Then $\gamma(K)=\gamma(\tilde{K})/|a_1|$ and the following diagram commutes, up to a multiplicative constant of modulus $1$:
\begin{diagram}
\Omega &\rTo^{F} &\tilde{\Omega}\\
 &\rdTo_{f} &\dTo_{\tilde{f}}\\
 & &\mathbb{C}
\end{diagram}
\end{proposition}

\begin{proof}
Note that $\tilde{f} \circ F \in H^{\infty}(\Omega)$ and $\|\tilde{f} \circ F\|_{\infty} \leq 1$. Thus,
$$
|(\tilde{f} \circ F)'(\infty)| \leq \gamma(K).
$$
However, we have
$$
(\tilde{f} \circ F)'(\infty)=\frac{\tilde{f}'(\infty)}{a_{1}}=\frac{\gamma(\tilde{K})}{a_{1}}
$$
and so
$$
\frac{\gamma(\tilde{K})}{|a_{1}|} \leq \gamma(K).
$$
Repeating this with $f \circ F^{-1}$ instead gives the reverse inequality.

By uniqueness of the Ahlfors function, we have $\tilde{f} \circ F = \lambda f$ for some constant $\lambda$ with $|\lambda|=1$.
\end{proof}

For the transformation of the Garabedian function $\psi$, we need additional assumptions on the boundary of $\Omega$. The reason behind this will be clear soon.

Accordingly, we shall assume that $F$ is $C^{\infty}$ on the boundary. This will be the case, for example, if all the boundary curves of $\Omega$ are $C^{\infty}$. This is a consequence of the following result, which dates back to Painlev\'e's doctoral thesis:

\begin{theorem}[Painlev\'e]
Let $D$ be a bounded Jordan domain with $C^{\infty}$ boundary, and let $f$ be a conformal mapping of $D$ onto the unit disk $\mathbb{D}$. Then $f$ is $C^{\infty}$ on $\overline{D}$, the derivative $f'$ does not vanish on $\overline{D}$, and $f^{-1}$ is $C^{\infty}$ on the closed unit disk.
\end{theorem}

\begin{proof}
For a proof, see e.g. \cite[Theorem 8.2]{BELL}.
\end{proof}

We shall also need the fact that $F'$ has an holomorphic square root in $\Omega$. This is a consequence of the following:

\begin{theorem}[Bell]
\label{sqrt}
Let $f: \Omega_1 \rightarrow \Omega_2$ be a conformal mapping between bounded finitely connected domains with $C^{\infty}$ boundaries. Then $f$ is $C^{\infty}$ on $\overline{\Omega}_1$ and $f'$ is nonvanishing on $\overline{\Omega}_1$. Consequently, $f^{-1}$ is $C^{\infty}$ on $\overline{\Omega}_2$. Furthermore, $f'$ is equal to the square of a function $C^{\infty}$ on $\overline{\Omega}_1$ and holomorphic in $\Omega_1$.
\end{theorem}

\begin{proof}
See \cite[Theorem 12.1]{BELL}.
\end{proof}

We can now prove:

\begin{theorem}
\label{propp}
Suppose that $\Omega$ is a finitely connected domain with $C^{\infty}$ boundary, and let $\tilde{\Omega}$ be as in the above.
Let $\tilde{\psi}$ be the Garabedian function for $\tilde{\Omega}$, as in Theorem \ref{Theo2}. Define a function $\psi$ in $\Omega$ by
$$
\psi:= \frac{1}{a_{1}}(\tilde{\psi} \circ F)F'.
$$
Then $\psi$ has the following properties:
\begin{enumerate}[\rm(i)]
 \item $\psi \in A(\Omega)$;
 \item $\psi(\infty) = 1/2 \pi i$;
 \item  $\psi$ represents evaluation of the derivative at $\infty$, in the sense that, for all $g \in A(\Omega)$,
$$g'(\infty) = \int_{\partial \Omega} g(\zeta)\psi(\zeta)d\zeta;$$
 \item $\int_{\partial \Omega} |\psi(\zeta)| |d\zeta| = \gamma(K)$;
\item  $\psi$ has an analytic square root in $\Omega$. More precisely, there exists a function $q \in A(\Omega)$ such that $q(\infty)=1$ and
$$
q(z)^2 = 2 \pi i \psi(z)
\qquad (z \in \overline{\Omega}).
$$
\end{enumerate}
\end{theorem}

\begin{proof}
(i) is clear. Indeed, $\psi$ is even $C^{\infty}$ on $\overline{\Omega}$.

To prove (ii), note that
$$
F'(z)=a_{1}-\frac{a_{-1}}{z^2}-\frac{2a_{-2}}{z^3}-\dots
$$
near $\infty$, so that $F'(z) \rightarrow a_{1}$ as $z \rightarrow \infty$. Since $F$ fixes $\infty$, we have
$$
(\tilde{\psi} \circ F)(\infty)=\tilde{\psi}(\infty)=\frac{1}{2\pi i}
$$
and (ii) follows.

To prove (iii), we use the change-of-variables formula as found in e.g. \cite[Theorem 7.26]{RUD}. Let $g \in A(\Omega)$. Then $g \circ F^{-1} \in A(\tilde{\Omega})$ so that, by Theorem \ref{Theo2},
\begin{align*}
(g \circ F^{-1})'(\infty)
&=  \int_{\partial \tilde{\Omega}} g(F^{-1}(w))\tilde{\psi}(w)dw \\
&=  \int_{\partial \Omega} g(z) \tilde{\psi}(F(z))F'(z)dz \\
&= a_{1}\int_{\partial \Omega} g(z) \psi(z)dz.
\end{align*}
Here we used the change of variable $w=F(z)$, which is legitimate since $F$ is injective and $C^{\infty}$ on $\partial \Omega$.
Now write
$$
g(z) = g(\infty) + \frac{g'(\infty)}{z}+ \frac{b_2}{z^2}+ \dots
$$
near $\infty$. Then
$$
(g \circ F^{-1})'(\infty)
= \lim_{w \rightarrow \infty} w(g(F^{-1}(w))-g(\infty))
= a_{1}g'(\infty)
$$
and (iii) follows.

To prove (iv), note that
\begin{align*}
\int_{\partial \Omega} |\psi(z)||dz|
&=  \frac{1}{|a_{1}|}\int_{\partial \Omega} |\tilde{\psi}(F(z))||F'(z)dz| \\
&=  \frac{1}{|a_{1}|}\int_{\partial \tilde{\Omega}} |\tilde{\psi}(w)||dw| \\
&= \frac{\gamma(\tilde{K})}{|a_{1}|} = \gamma(K),
\end{align*}
by Proposition \ref{prop1}.

Finally, (v) follows directly from Theorem \ref{Theo2} and the remark preceding Theorem \ref{sqrt}.
\end{proof}

Adopting the terminology already used in the case where $\Omega$ had analytic boundary,
from now on we shall call $\psi$ the Garabedian function for $\Omega$.

\begin{remark} The Garabedian function was
studied by Garnett \cite{GAR} and Havinson \cite{HAV}, who both raised the question of whether
the Garabedian functions of a decreasing sequence of compact sets with analytic
boundaries must converge. This question was answered in the affirmative by Smith
\cite{SMIT} and also by Suita \cite{SUI2}. This fact leads to a natural definition of the Garabedian
function for an arbitrary compact plane set, though we shall not need this degree of generality here.
\end{remark}

\section{Estimates for analytic capacity in the case of \protect{$C^{\infty}$} boundary}
\label{sec3}
In this section, we obtain some estimates for the analytic capacity of a compact set whose complement is a finitely connected domain with $C^{\infty}$ boundary. First, we need a lemma:

\begin{lemma}
\label{cauchy1}
Let $K$ be a compact set in the plane whose complement $\Omega$ is a finitely connected domain with $C^{\infty}$ boundary. If $f \in A(\Omega)$, then
$$
f'(\infty) = \frac{1}{2 \pi i} \int_{\partial \Omega} f(\zeta) d\zeta.
$$
\end{lemma}

\begin{proof}
Write
$$
f(z) = f(\infty) + \frac{f'(\infty)}{z} + \frac{a_2}{z^2} + \dots
$$
near $\infty$. Let $C$ be a circle centered at the origin and containing $K$,  with radius sufficiently large so that the above expression for $f$ holds on $C$. Then we have
$$
\frac{1}{2\pi i} \int_{\partial \Omega} f(\zeta) d\zeta = \frac{1}{2 \pi i}\int_{C} f(\zeta) d\zeta.
$$
Indeed, the above is clear if $f$ is holomorphic in a neighborhood of $\overline{\Omega}$, and such functions are uniformly dense in $A(\Omega)$, by Mergelyan's theorem.

On the other hand, the right-hand side in the last expression is just $f'(\infty)$. This can be seen by substituting the power series expression for $f$ into the integral and integrating term by term.
\end{proof}

Now we can prove:
\begin{theorem}
\label{est1}
Let $K$ be a compact set in the plane, and suppose that the complement $\Omega$ of $K$ is a finitely connected domain with $C^{\infty}$ boundary. Then
\begin{equation}
\label{estime1}
\gamma(K) = \min \left\{ \frac{1}{2\pi}\int_{\partial \Omega} |g(z)|^2|dz|: g \in A(\Omega), \, g(\infty)=1 \right\}
\end{equation}
and
\begin{equation}
\label{estime2}
\gamma(K) = \max \left\{ 2\Re h'(\infty) - \frac{1}{2\pi}\int_{\partial \Omega} |h(z)|^2 |dz|: h \in A(\Omega), \, h(\infty)=0 \right\}.
\end{equation}
Here the minimum and maximum are attained respectively by the functions $g=q$ and $h=fq$, where $f$ is the Ahlfors function for $K$ and $q$ is the function of Theorem \ref{propp}.
\end{theorem}

The identity (\ref{estime1}) was already known in the case of analytic boundary, since the work of Garabedian \cite{GARA}. It is usually referred to as \textit{Garabedian's duality}. It was also studied for more general domains $\Omega$ by Havinson~\cite{HAV}.

\begin{proof}
Let $f$ be the Ahlfors function for $\Omega$, so that $f \in A(\Omega)$ with $f(\infty)=0$ and $|f| \equiv 1$ on $\partial \Omega$ and $f'(\infty)=\gamma(K)$. Let $\psi$ be the Garabedian function for $\Omega$ in the sense of Theorem \ref{propp}, and denote by $q$ the function in $A(\Omega)$ with $q(\infty)=1$ and
$$
q(z)^2 = 2\pi i \psi(z) \qquad (z \in \overline{\Omega} ).
$$

To prove (\ref{estime1}), let $g \in A(\Omega)$ with $g(\infty)=1$. We have
$$
\gamma(K) = f'(\infty) = (fg^2)'(\infty)
= \frac{1}{2\pi i} \int_{\partial \Omega} f(z)g(z)^2 dz
$$
by Lemma \ref{cauchy1}.
Thus,
$$
\gamma(K) \leq \frac{1}{2\pi}\int_{\partial \Omega} |f(z)| |g(z)|^2 |dz|
=  \frac{1}{2\pi}\int_{\partial \Omega}|g(z)|^2 |dz|.
$$
Taking the minimum over all such functions $g$, we obtain
\begin{equation}
\label{eqtom1}
\gamma(K) \leq \min \left\{ \frac{1}{2\pi}\int_{\partial \Omega} |g(z)|^2|dz|: g \in A(\Omega): \, g(\infty)=1 \right\}.
\end{equation}

On the other hand, take $g:= q$. Then $g \in A(\Omega)$ with $g(\infty)=1$ and
$$
\frac{1}{2\pi}\int_{\partial \Omega} |g(z)|^2|dz|
= \frac{1}{2\pi}\int_{\partial \Omega} |q(z)|^2|dz|
= \int_{\partial \Omega} |\psi(z)||dz| = \gamma(K),
$$
by Theorem \ref{propp}. Combining this with inequality (\ref{eqtom1}), we obtain (\ref{estime1}).

For (\ref{estime2}), consider the function $h=fq$. Then $h \in A(\Omega)$, $h(\infty)=0$ and
$$
2 \Re h'(\infty) - \frac{1}{2\pi}\int_{\partial \Omega} |h(z)|^2 |dz|
= 2\gamma(K) - \gamma(K) = \gamma(K).
$$
Thus
\begin{equation}
\label{eqtom2}
\max \left\{ 2\Re h'(\infty) - \frac{1}{2\pi}\int_{\partial \Omega} |h(z)|^2 |dz|: h \in A(\Omega) :\, h(\infty)=0 \right\}
\geq \gamma(K).
\end{equation}

On the other hand, let $h \in A(\Omega)$, $h(\infty)=0$.
Denote by $T(z)$ the unit tangent vector to $\partial \Omega$ at $z$,
that is, $dz = T(z) |dz|$ with $|T| \equiv 1$.
Let $\langle h_1, h_2 \rangle$ denote
$\int_{\partial \Omega} h_1(z) \overline{h_2(z)} |dz|$
and $\|h\|_2^2:=\langle h,h \rangle$.
Then
$$
0 \leq \frac{1}{2\pi}\|h-i\overline{qT}\|_{2}^2
= \frac{1}{2\pi} \|h\|_{2}^2 + \frac{1}{2\pi} \|q\|_{2}^2 + 2\Re \frac{1}{2\pi} \langle h,-i\overline{qT} \rangle,
$$
so that
$$
0 \leq  \frac{1}{2\pi}\int_{\partial \Omega} |h(z)|^2|dz|
+ \gamma(K)  -2 \Re  \frac{1}{2\pi i}\int_{\partial \Omega} h(z)q(z)dz.
$$
By Lemma \ref{cauchy1}, it follows that
$$
0 \leq  \frac{1}{2\pi}\int_{\partial \Omega} |h(z)|^2|dz|
+  \gamma(K)  -2 \Re (hq)'(\infty).
$$
Since $q(\infty)=1$ and $h(\infty)=0$, we have $(hq)'(\infty) = h'(\infty)$, and thus
$$
\gamma(K) \geq   2 \Re h'(\infty) - \frac{1}{2\pi}\int_{\partial \Omega} |h(z)|^2 |dz|.
$$
Combining this with inequality (\ref{eqtom2}), we obtain (\ref{estime2}).
\end{proof}

\section{Estimates for analytic capacity in the case of  piecewise-analytic boundary}
\label{sec4}

The objective of this section is to extend the estimates of Theorem \ref{est1} to another interesting case,
that of sets with piecewise-analytic boundaries.

Let us assume that $K$ is a compact set such that $\Omega:= \mathbb{C}_{\infty} \setminus K$ is a finitely connected domain with piecewise-analytic boundary. By this, we mean that the boundary consists of a finite number of non-intersecting Jordan curves, and that each boundary curve is the union of a finite number of analytic arcs. We further assume that every intersecting pair of analytic arcs meets at a corner that is conformally equivalent to a sector. More precisely, we suppose that if two analytic arcs intersect at a point $w$, then there exists a conformal map defined in a neighborhood $V$ of $w$, and mapping $V\cap\Omega$ onto a sector $\{re^{i\theta}:0<r<1,~0<\theta<\alpha\}$, where $0<\alpha<2\pi$.

To obtain the estimates of Theorem \ref{est1}, we used the fact that the functions $q$ and $fq$ extend continuously to the boundary, where $f$ is the Ahlfors function for $K$ and $q$ is the square root of $2 \pi i$ times the Garabedian function $\psi$.

In the case of piecewise-analytic boundary, this remains true for the Ahlfors function $f$. However, there are some issues regarding the Garabedian function: $\psi$ will have discontinuities at the finite set $E$ where the boundary fails to be smooth ($E$ is made of the endpoints of the analytic arcs in the boundary). Consequently, in order to extend Theorem \ref{est1}, we need to replace $A(\Omega)$ by a larger class of holomorphic functions in $\Omega$. It turns out that Smirnov classes over finitely connected domains are precisely what we need.

\subsection{Smirnov classes on finitely connected domains}

Our objective now is to present the theory of Smirnov classes $E^p(\Omega)$ on finitely connected domains.
For more details, we refer the reader to \textbf{\cite{DUR}} and \textbf{\cite{SMIR}}.

Let $\Omega$ be a finitely connected domain with rectifiable boundary. By that, we mean that $\partial \Omega$ consists of a finite number of pairwise disjoint rectifiable Jordan curves. Let $1 \leq p < \infty$.

We say that a function $h$ belongs to the \textit{Smirnov class} $E^p(\Omega)$ if $h$ is holomorphic in $\Omega$ and if there exists a sequence $\{\Omega_n\}$ of finitely connected subdomains of $\Omega$ with rectifiable boundaries $\{C_n\}$ such that:
\begin{enumerate}[\rm(i)]
\item $\Omega_n$ eventually contains each compact subset of $\Omega$,
\item $\limsup_{n\rightarrow \infty}\int_{C_n}|h(z)|^p|dz| < \infty,$
\item the lengths of the curves of the $C_n$'s are uniformly bounded.
\end{enumerate}

In the simply connected case, it is well known that condition (iii) is a superfluous requirement in the definition. This remains true in the finitely connected case.

It is not hard to prove that $A(\Omega) \subseteq E^p(\Omega)$.
Moreover, it is also well known that $E^p(\Omega)$ reduces to the classical Hardy space $H^p(\Omega)$ when $\Omega$ is the unit disk, and this is also true if $\Omega$ is a finitely connected domain with analytic boundary. However, in general, neither of the inclusions hold. Even for simple cases like Jordan domains with polygonal boundaries, the two classes are not equal.

The following well-known result is a generalization of Fatou's theorem on the classical Hardy spaces.

\begin{theorem}
\label{cauchy}
Let $p \geq 1$ and let $\Omega$ be a bounded finitely connected domain with rectifiable boundary. Suppose that $h \in E^p(\Omega)$. Then
\begin{enumerate}[\rm(i)]
\item  $h$ has nontangential boundary values $h^{*}$ almost everywhere on $\partial \Omega$, and $h^{*} \in L^p(\partial \Omega)$.
\item $h$ is the Cauchy integral of $h^{*}$:
$$
h(z) = \frac{1}{2\pi i} \int_{\partial \Omega} \frac{h^{*}(\zeta)}{\zeta-z}d\zeta
\qquad (z \in \Omega).
$$
\end{enumerate}
\end{theorem}

We shall also need the following generalization of Lemma \ref{cauchy1}:

\begin{corollary}
\label{cor1}
Let $K$ be a compact set in the plane, and suppose that the complement $\Omega$ of $K$ is a finitely connected domain with rectifiable boundary. Let $h \in E^1(\Omega)$. Then
$$
h'(\infty) = \frac{1}{2\pi i} \int_{\partial \Omega} h^{*}(\zeta)d\zeta.
$$
\end{corollary}

\begin{proof}
Translating $K$, we may suppose that it contains $0$ in its interior. Set $\Omega_0:= \{z: 1/z \in \Omega\}$. Then $\Omega_0$ is a bounded finitely connected domain with rectifiable boundary, and $0 \in \Omega_0$. Define $h_0$ in $\Omega_0$ by $h_0(z)=h(1/z).$ It is easy to see that $h_0 \in E^1(\Omega_0)$, since if $\Gamma_0$ is any rectifiable Jordan curve in $\Omega_0$ not passing through $0$, then
$$
\int_{\Gamma_0}|h_0(z)||dz| = \int_{\Gamma}|h(w)|\frac{|dw|}{|w|^2},
$$
where $\Gamma:= \{z: 1/z \in \Gamma_0\}$, and the function $1/w^2$ is bounded in $\Omega$.
Now, the function $g_0(z):=(h_0(z)-h_0(0))/z$ clearly also belongs to $E^1(\Omega_0)$, so that
$$
h'(\infty)=g_0(0)=\frac{1}{2\pi i}\int_{\partial \Omega_0}\frac{h_0^{*}(z)-h_0(0)}{z^2} dz
= \frac{1}{2\pi i}\int_{\partial \Omega_0}\frac{h_0^{*}(z)}{z^2}dz,
$$
where we used the preceding theorem. Making the change of variable $w=1/z$, we obtain the result.
\end{proof}

\subsection{The Garabedian function}

Let us now return to the case where $K$ is a compact set in the plane whose complement $\Omega$ is a finitely connected domain with piecewise-analytic boundary. In this case, there are some issues regarding the Garabedian function $\psi$. If we proceed as in Theorem \ref{propp} and define
$$
\psi:= \frac{1}{a_{1}}(\tilde{\psi} \circ F)F',
$$
then $\psi$ will not extend continuously to the boundary: the first factor $\tilde{\psi} \circ F$ is in $A(\Omega)$, but the second one $F'$ has singularities at the endpoints of the analytic arcs, i.e. at the points of $E$. However, $F$ extends analytically across any analytic arc in the boundary, so $\psi$ as defined is continuous in $\overline{\Omega} \setminus E$.

The following result is the analogue of Theorem \ref{propp} in this new setting. Recall that $\tilde{\Omega}$ is a finitely connected domain with analytic boundary conformally equivalent to $\Omega$, and that $F: \Omega \rightarrow \tilde{\Omega}$ is a conformal map, normalized so that $F(\infty)=\infty$, and with expansion
$$
F(z)=a_1z + a_0 + \frac{a_{-1}}{z} + \frac{a_{-2}}{z^2} + \dots
$$
near infinity.

\begin{theorem}
\label{propp2}
Let $\tilde{\psi}$ be the Garabedian function for $\tilde{\Omega}$. Define a function $\psi$ in $\Omega$ by
$$
\psi:= \frac{1}{a_{1}}(\tilde{\psi} \circ F)F'.
$$
Then $\psi$ is a Garabedian function for $\Omega$, in the sense that
\begin{enumerate}[\rm(i)]
 \item   $\psi$ is holomorphic in $\Omega$ and continuous in $\overline{\Omega} \setminus E$,
 \item   $\psi(\infty) = 1/2 \pi i$,
 \item   $\psi$ represents evaluation of the derivative at $\infty$, in the sense that for all $g \in A(\Omega)$,
$$
g'(\infty) = \int_{\partial \Omega} g(\zeta)\psi(\zeta)d\zeta,
$$
 \item   $\int_{\partial \Omega} |\psi(\zeta)| |d\zeta| = \gamma(K)$,
\item   $\psi$ has an analytic square root in $\Omega$. More precisely, there exists a function $q$ holomorphic in $\Omega$ and continuous in $\overline{\Omega} \setminus E$ such that $q(\infty)=1$ and
$$
q(z)^2 = 2 \pi i \psi(z) \qquad (z \in \overline{\Omega} \setminus E).
$$
\item $\psi \in E^1(\Omega)$. Consequently, $q \in E^2(\Omega)$.
\end{enumerate}
\end{theorem}

For the proof, we need an analogue of Theorem \ref{sqrt}:

\begin{lemma}
\label{sqrt2}
Let $\Omega, \tilde{\Omega}$ and $F: \Omega \rightarrow \tilde{\Omega}$ be as in the above. Then $F'$ has an analytic square root in $\Omega$. More precisely, there exists a function $h$ holomorphic in $\Omega$ and continuous in $\overline{\Omega} \setminus E$ such that
$$
h(z)^2 = F'(z) \qquad (z \in \overline{\Omega} \setminus E).
$$
\end{lemma}

\begin{proof}
The proof of \cite[Theorem 12.1]{BELL} also works in our case. However, we can use the fact that $F$ is a composition of Riemann maps to obtain a more elementary proof, as follows.

Let $n$ be the number of curves in the boundary of $\Omega$. Recall that by construction, $F$ is a composition of $n$ Riemann maps:
$$
F=\phi_1 \circ \phi_2 \circ \dots \circ \phi_n,
$$
where each $\phi_j$ maps some unbounded Jordan domain onto $\mathbb{C}_{\infty} \setminus \mathbb{D}$,
with $\phi_j(\infty)=\infty$.

We proceed by induction on $n$.

First, consider the case $n=1$.
Translating $\Omega$ if necessary, we can suppose that $0 \notin \overline{\Omega}$.
Put $D_1:= \{ z: 1/z \in \Omega \}$ and define
$$
G(z):=\frac{1}{F(1/z)} \qquad (z \in D_1).
$$
Then $G$ is a conformal mapping of the bounded Jordan domain $D_1$ onto $\mathbb{D}$ with $G(0)=0$. We know that $G$ extends to a homeomorphism of $\overline{D_1}$ onto $\mathbb{D}$ and analytically across any analytic arc of the boundary. Note that the boundary of $D_1$ consists of a finite number of analytic arcs separated by a finite set of points. Call this finite set of points $\tilde{E}$. It is easy to see that $G$ extends analytically to a simply connected domain $U$ containing $\overline{D_1} \setminus \tilde{E}$, with $G' \neq 0$ there.
Write $G'(z)=g(z)^2$ for some $g$ holomorphic in $U$. We have
$$
F(z)=\frac{1}{G(1/z)} \qquad (z \in \Omega)
$$
so that
$$
F'(z)= \frac{G'(1/z)}{z^2G(1/z)^2}
= \Bigl(\frac{g(1/z)}{zG(1/z)}\Bigr)^2:=h(z)^2
\qquad (z \in \Omega).
$$
But both $F'$ and $h$ are holomorphic in a neighborhood of $\overline{\Omega} \setminus E$. In particular, we have
$$
F'(z)=h(z)^2 \qquad (z \in \overline{\Omega} \setminus E).
$$
This completes the proof for the case $n=1$.

Now, suppose that the result holds for $n-1$, where $n \geq 2$.
Since $F=\phi_1 \circ \phi_2 \circ \dots \circ \phi_n$, we have
$$
F'(z) = \phi_1'((\phi_2 \circ \dots \circ \phi_n)(z)) (\phi_2 \circ \dots \phi_n)'(z)
$$
and the result follows from the induction hypothesis and the case $n=1$.
\end{proof}

We can now proceed to the proof of Theorem \ref{propp2}:

\begin{proof}
We already know that (i) holds, and the proof of (ii) is exactly the same than the one in Theorem \ref{propp}. Moreover, since $F$ is differentiable everywhere on the boundary except at a finite set of points, we can use the change-of-variables formula found e.g. in \cite[Theorem 7.26]{RUD}. Points (iii) and (iv) then follow exactly as in the proof of Theorem \ref{propp}.

Point (v) follows directly from Theorem \ref{Theo2}, together with Lemma \ref{sqrt2}.

For (vi), since $\psi:= \frac{1}{a_{1}}(\tilde{\psi} \circ F)F'$ and $\tilde{\psi} \circ F$ is bounded in $\Omega$, it suffices to show that $F' \in E^1(\Omega)$. But this is clear, since if $C$ is any rectifiable Jordan curve in $\Omega$, then
$$
\int_{C} |F'(z)||dz| = \int_{F(C)} |dw|.
$$
\end{proof}

\subsection{Proof of the estimates}

We can now prove:

\begin{theorem}
\label{est}
Let $K$ be a compact set in the plane,
and suppose that the complement $\Omega$ of $K$ is
a finitely connected domain with piecewise-analytic boundary. Then
\begin{equation*}
\gamma(K)
= \min \left\{ \frac{1}{2\pi}\int_{\partial \Omega} |g^{*}(z)|^2|dz|: g \in E^2(\Omega), \, g(\infty)=1 \right\}
\end{equation*}
and
\begin{equation*}
\gamma(K)
= \max \left\{ 2 \Re h'(\infty) - \frac{1}{2\pi}\int_{\partial \Omega} |h^{*}(z)|^2 |dz|: h \in E^2(\Omega), \, h(\infty)=0 \right\}.
\end{equation*}
Here the minimum and maximum are attained respectively by the functions $g=q$ and $h=fq$, where $f$ is the Ahlfors function for $K$ and $q$ is the function of Theorem~\ref{propp2}.
\end{theorem}

\begin{proof}
The proof is exactly the same as the one in Theorem \ref{est1}. Use Theorem \ref{propp2} instead of Theorem \ref{propp} and Corollary \ref{cor1} instead of Lemma \ref{cauchy1}.
\end{proof}

\begin{remark}
It follows from the first estimate that the function $q$ is unique and, consequently, the Garabedian function $\psi$ too. Indeed, $q$ is an element of minimal norm in the convex set $S:=\{g \in E^2(\Omega): g(\infty)=1\}$, which is necessarily unique by an elementary Hilbert-space argument.
\end{remark}

\section{Computation of analytic capacity}
\label{sec5}

\subsection{Description of the method}

In this section, we present a method based on the estimates of Theorem \ref{est1} (respectively Theorem \ref{est}) to compute the analytic capacity of a compact set $K$ whose complement $\Omega$ is a finitely connected domain with analytic (respectively piecewise-analytic) boundary. The method yields upper and lower bounds for $\gamma(K)$.

Let $A_0(\Omega):=\{ f \in A(\Omega): f(\infty)=0 \}$ and let $\mathcal{F}$ be a subset of $A_0(\Omega)$ whose span is dense, with respect to the $L^2$-norm on $\partial \Omega$. For example, $\mathcal{F}$ could be the set of all functions of the form $(z-a)^{-n}$, where $n \in \mathbb{N}$ and $a$ belongs to some prescribed set $S$ containing one point in each component of the interior of $K$. This is a consequence of Mergelyan's theorem.

Let $\mathcal{A}=\{g_1, g_2, \dots, g_n\}$ be a finite subset of $\mathcal{F}$. The functions $g_1, g_2, \dots, g_n$ will be called \textit{approximating functions}.

The method for the upper bound is based on the following:
\begin{itemize}
\item Finding the function $g$ in the span of $g_1, g_2, \dots, g_n$ that minimizes the quantity
$$
\frac{1}{2\pi}\int_{\partial \Omega} |1+g(z)|^2|dz|.
$$
\end{itemize}

In view of Theorems \ref{est1} and \ref{est}, this gives an upper bound for $\gamma(K)$.

More precisely, the method is the following.

\begin{itemize}
\item Define
$$
g(z):= \sum_{j=1}^n \alpha_j g_j(z),
$$
where $\alpha_1, \dots, \alpha_n$ are complex numbers to determine. Write $\alpha_j:=c_j + id_j$, and then compute the integral
$$
\frac{1}{2\pi}\int_{\partial \Omega} |1+g(z)|^2|dz|.
$$
This gives an expression that can be written in the form $\frac{1}{2}\mathbf{x^T}\mathbf{A}\mathbf{x} + \mathbf{b}\mathbf{x} + c$, where $\mathbf{x}=(c_1, c_2, \dots, c_n, d_1, d_2, \dots, d_n)$, and where $\mathbf{A}$ is a real symmetric positive-definite $2n \times 2n$ matrix, $\mathbf{b}$ is a real vector of length $2n$ and $c$ is a positive constant.
\item Find the $c_j$'s and $d_j$'s that minimize this expression. This can be done for example by solving the linear system
 $$
 \mathbf{A}\mathbf{x}+\mathbf{b}=0.
 $$
\item Create a new set of approximating functions $\tilde{\mathcal{A}}$ by adding functions from $\mathcal{F}$ to $\mathcal{A}$, and then repeat the procedure with $\mathcal{A}$ replaced by $\tilde{\mathcal{A}}$.
\end{itemize}

The above yields a sequence of decreasing upper bounds for $\gamma(K)$.
Clearly, it can be adapted to yield a sequence of increasing lower bounds for $\gamma(K)$, using the other estimate of Theorems \ref{est1} and \ref{est}.

\subsection{Convergence of the method}

In this subsection, we prove that the upper and lower bounds obtained with the method can in principle be made arbitrarily close.

First, recall that the minimum and maximum in Theorems \ref{est1} and \ref{est} are attained respectively by the functions $q$ and $fq$, where $f$ is the Ahlfors function for $K$ and $q$ is the square root of $2\pi i$ times the Garabedian function for $\Omega$.

If the boundary of $\Omega$ is $C^{\infty}$, then both of these functions belong to $A(\Omega)$. It follows from Mergelyan's theorem that we can approximate them uniformly on $\partial \Omega$ by rational functions with poles in $S$, where $S$ is some prescribed set containing at least one point in each component of the interior of $K$. On the other hand, Lemma \ref{cauchy1} implies that if $h,R \in A(\Omega)$ and $|h-R| < \epsilon$ on $\partial \Omega$, then $|R'(\infty)-h'(\infty)|<C\epsilon$ where $C$ depends only on $\Omega$.
Hence, it follows that for every $\epsilon>0$, there exist rational functions $R_1, R_2$ vanishing at $\infty$ and with poles in the prescribed set $S$, such that
$$
\frac{1}{2\pi}\int_{\partial \Omega} |1+R_2(z)|^2 |dz| - \epsilon
\leq \gamma(K) \leq 2  \Re  R_1'(\infty) - \frac{1}{2\pi}\int_{\partial \Omega} |R_1(z)|^2 |dz| + \epsilon.$$
This proves the convergence of the bounds in the $C^{\infty}$ boundary case.

For the piecewise-analytic boundary case, we need an analogue of Mergelyan's Theorem for the Smirnov class $E^2(\Omega)$.
Assume that $K$ is a compact set in the plane whose complement is a finitely connected domain with piecewise-analytic boundary.
We know that $E^2(\Omega)$ contains $A(\Omega)$, but is it true that $A(\Omega)$ is \textit{dense} in $E^2(\Omega)$? In other words, can every function $h$ in $E^2(\Omega)$ be approximated on the boundary by functions $f_n$ in $A(\Omega)$, in the sense that
$$
\int_{\partial \Omega}|h^{*}(z)-f_n(z)|^2 |dz| \rightarrow 0
$$
as $n \rightarrow \infty$?
In turns out that the answer is yes. Before we prove this, we need the definition of \textit{Smirnov domains}:

Let $D \subseteq \mathbb{C}$ be a bounded Jordan domain with rectifiable boundary. Since $D$ is simply connected, there is a conformal mapping $\phi$ of the open unit disk $\mathbb{D}$ onto $D$. It is well known that $\phi'$ is in $H^1(\mathbb{D})$, and, since it has no zeros, we have a canonical factorization of the form
$$
\phi'(z) = S(z)Q(z) \qquad (z \in \mathbb{D}),
$$
where $S$ is a singular inner function and $Q$ is outer. We say that $D$ is a \textit{Smirnov domain} if $S \equiv 1$, that is, if $\phi'$ is outer.
It can be shown that this definition is independent of the function $\phi$; it depends only on the domain $D$. A simple sufficient condition for $D$ to be a Smirnov domain is that $\arg{\phi'}$ be bounded either from above or below. Geometrically, this means that the local rotation of the mapping is bounded; loosely speaking, the boundary curve cannot spiral too much. In particular, $D$ is a Smirnov domain if it has smooth (or piecewise-smooth) boundary. We refer the reader to \textbf{\cite{DUR}} or \cite[Chapter 7]{POM} for more details on Smirnov domains.

We say that a function $h \in L^p(\partial D)$ belongs to the $L^p(\partial D)$-closure of the polynomials if there is a sequence $(p_n)$ of polynomials such that
$$
\int_{\partial D}|h(z)-p_n(z)|^p|dz| \rightarrow 0 \qquad (n \rightarrow \infty).
$$
It is convenient to identify $E^p(D)$ with its set of boundary values functions. Thus, $E^p(D)$ is a closed subspace of $L^p(\partial D)$ which contains the polynomials, hence also their closure. For the reverse inclusion, we have the following criterion:

\begin{theorem}
\label{smirn}
Let $D$ be a bounded Jordan domain with rectifiable boundary, and let $1 \leq p < \infty$. Then $E^p(D)$ coincides with the $L^p(\partial D)$-closure of the polynomials if and only if $D$ is a Smirnov domain.
\end{theorem}

\begin{proof}
See e.g. \textbf{\cite[Theorem 10.6]{DUR}}.
\end{proof}

Our objective is to generalize Theorem \ref{smirn} to finitely connected domains. We shall need the following theorem:

\begin{theorem}[Decomposition Theorem]
\label{decomp}
Suppose that $D$ is a bounded finitely connected domain whose boundary consists of pairwise disjoint rectifiable Jordan curves $\Gamma_1, \Gamma_2, \dots, \Gamma_n$, where the outer boundary of $D$ is $\Gamma_1$. For $1 \leq j \leq n$, let $D_j$ be the component of $\mathbb{C}_{\infty} \setminus \Gamma_j$ that contains $D$.
Let $1 \leq p < \infty$, and let $h \in E^p(D)$. Then $h$ can be decomposed uniquely as
$$
h(z) = h_1(z)+h_2(z)+\dots+h_n(z) \qquad (z \in D),
$$
where each $h_j$ belongs to $E^p(D_j)$ and $h_j(\infty)=0$ for $2 \leq j \leq n$.
\end{theorem}

\begin{proof}
See \textbf{\cite{TUM}}.
\end{proof}

The following is a generalization of Theorem \ref{smirn} to finitely connected domains, in the case $p=2$:

\begin{theorem}
\label{theonice}
Let $D$ and $D_j$ $(1 \leq j \leq n)$ be as in Theorem \ref{decomp}, and suppose in addition that the curves $\Gamma_j$ are piecewise analytic.  Let $a_1:=\infty$ and for $2 \leq j \leq n$, fix a point $a_j$ in the interior of the complement of $D_j$.
Then the rational functions with poles in the set $\{a_1, a_2, \dots, a_n\}$ are dense in $E^2(D)$. In other words, for every $h \in E^2(D)$, there exists a sequence $(R_n)$ of rational functions with poles in the prescribed set $\{a_1, a_2, \dots, a_n\}$ such that
$$
\int_{\partial D}|h^{*}(z)-R_n(z)|^2 |dz| \rightarrow 0
$$
as $n \rightarrow \infty$.
\end{theorem}

For the proof, we need the following lemma:

\begin{lemma}
\label{lem22}
Let $1 \leq p< \infty$, let $U$ be a  bounded Jordan domain with rectifiable boundary and let $K$ be a compact subset of $U$. Then there is a constant $M$, depending only on $p$ and $K$, such that
$$
|g(w)| \leq M \|g^{*}\|_p
$$
for all $w \in K$ and for every function $g \in E^p(U)$, where
$$
\|g^{*}\|_p:= \Bigl(\int_{\partial U}|g^{*}(z)|^p|dz|\Bigr)^{1/p}.
$$
\end{lemma}

\begin{proof}
This is a simple application of Theorem \ref{cauchy} and H\"older's inequality.
\end{proof}

We can now prove the theorem:

\begin{proof}[Proof of Theorem \ref{theonice}]
By the decomposition theorem, it suffices to show that, if $h \in E^2(D_j)$, then there exists a rational function $R_j$ with poles only at $a_j$, such that the integral
$$
\int_{\partial D}|h^{*}(z)-R_j(z)|^2 |dz|
$$
can be made arbitrarily small. Let $\epsilon>0$.

First consider the case $j=1$, so that $h \in E^2(D_1)$. Since $D_1$ is a bounded Smirnov domain with boundary $\Gamma_1$, we can apply Theorem \ref{smirn} and find a polynomial $P_1$ such that
$$
\int_{\Gamma_1}|h^{*}(z)-P_1(z)|^2 |dz| < \epsilon.
$$
Now, since the function $(h-P_1) \in E^2(D_1)$, we know by Lemma \ref{lem22} that there exists a constant $M$, depending only on the curves $\Gamma_2, \Gamma_3, \dots, \Gamma_n$, such that
$$
|h(w)-P_1(w)|^2 \leq M \int_{\Gamma_1}|h^{*}(z)-P_1(z)|^2 |dz|
< M\epsilon \qquad (w \in \Gamma_2 \cup \dots \cup \Gamma_n).
$$
Thus, we have
\begin{align*}
\int_{\partial D}|h^{*}(z)-P_1(z)|^2 |dz|
&= \int_{\Gamma_1}|h^{*}(z)-P_1(z)|^2 |dz| + \sum_{j=2}^{n}\int_{\Gamma_j}|h(z)-P_1(z)|^2 |dz| \\
&< \epsilon(1+ML),
\end{align*}
where $L$ is the sum of the lengths of the curves $\Gamma_2$, $\Gamma_3$, $\dots$, $\Gamma_n$. Since the right side can be made arbitrarily small, this concludes the proof for the case $j=1$.

Suppose now that $2 \leq j \leq n$ and let $h \in E^2(D_j)$. We can suppose $h(\infty)=0$, since it is part of the conclusion in the decomposition theorem. Translating $D_j$, we may suppose that $a_j=0$. Let $\tilde{D}_j:=\{z:1/z \in D_j\}$. Then $\tilde{D}_j$ is a bounded Jordan domain with piecewise-analytic boundary; in particular it is a Smirnov domain. Define a function $g$ in $\tilde{D}_j$ by $g(z):=h(1/z)$. It is easy to check that $g \in E^2(\tilde{D}_j)$. Also, $g$ vanishes at $0$, so the function  $g(z)/z$ also belongs to $E^2(\tilde{D}_j)$. We have, by Theorem \ref{smirn},
$$
\int_{\partial \tilde{D}_j}\Bigl|\frac{g^{*}(z)}{z}-P(z)\Bigr|^2 |dz| < \epsilon
$$
for some polynomial $P$. Making the change of variable $w=1/z$, we obtain
$$
\int_{\partial D_j}\Bigl|wh^{*}(w)-P\Bigl(\frac{1}{w}\Bigr)\Bigr|^2 \frac{1}{|w|^2}|dw| < \epsilon.
$$
Let $Q(w):=(1/w)P(1/w)$, so that $Q$ is a rational function with poles only at $0$. The above inequality becomes
$$
\int_{\partial D_j}\left|h^{*}(w)-Q(w)\right|^2|dw| < \epsilon.
$$
Now, apply Lemma \ref{lem22} once again to conclude that there is a constant $M$ depending only on the curves $\Gamma_k$ for $k \neq j$, such that
$$
\Bigl|\frac{g(z)}{z}-P(z)\Bigr|^2 \leq M \epsilon$$
for all $z \in \tilde{\Gamma}_k$, $k \neq j$, where $\tilde{\Gamma}_k:=\{w:1/w \in \Gamma_j\}$. Thus,
\begin{align*}
\int_{\partial D}|h^{*}(w)-Q(w)|^2 |dw|
&= \int_{\Gamma_j}|h^{*}(w)-Q(w)|^2 |dw| + \sum_{k \neq j}\int_{ \Gamma_k}|h(w)-Q(w)|^2 |dw| \\
&< \epsilon + \sum_{k \neq j}\int_{\tilde{\Gamma}_k}\Bigl|\frac{g(z)}{z}-P(z)\Bigr|^2 |dz| \\
&< \epsilon + ML'\epsilon = (1+ML')\epsilon,
\end{align*}
where $L'$ is the sum of the lengths of the curves $\tilde{\Gamma}_k$, $k \neq j$. Since the right side can be made arbitrarily small, this completes the proof of the theorem.
\end{proof}

\begin{remark}
In the above proof, piecewise-analyticity of the boundary is only assumed so that the domains $D_1$ and $\tilde{D}_j$ $(2 \leq j \leq n)$ are Smirnov domains. The result therefore remains true under this weaker assumption.
\end{remark}

Now, let us apply Theorem \ref{theonice} to our case: i.e. $K$ is a compact set in the plane whose complement $\Omega$ is a finitely connected domain with piecewise-analytic boundary consisting of $n$ curves. Let $D$ be the bounded domain obtained by adding an outer boundary circle, say $\Gamma_0$, with radius sufficiently large so that $K$ is contained in the interior $D_0$ of $\Gamma_0$. If $h \in E^2(\Omega)$, then clearly $h \in E^2(D)$. The decomposition theorem gives
$$
h=h_0+h_1+ \dots + h_n
$$
where each $h_j$ belongs to $E^2(D_j)$. Thus,
$$
h_0=h-h_1- \dots -h_n,
$$
and, since $h$ is holomorphic in the complement of $K$, this gives an analytic extension of $h_0$ to the entire plane, bounded near $\infty$. By Liouville's theorem, $h_0$ is constant.

Now, since $D$ is a bounded finitely connected domain with piecewise-analytic boundary, we can apply Theorem \ref{theonice} to $h$ on $D$. Since $h_0$ is constant, we can omit the pole at $\infty$. Thus, we have proved:

\begin{corollary}
Suppose that $K$ is a compact set in the plane whose complement $\Omega$ is a finitely connected domain with piecewise-analytic boundary consisting of $n$ curves. For each $1 \leq j \leq n$, fix a point $a_j$ in the interior of each component of $K$. Then the rational functions with poles in the prescribed set $\{a_1, a_2, \dots, a_n\}$ are dense in $E^2(\Omega)$.
\end{corollary}

\newpage

\section{Numerical examples}
\label{sec6}
In this section, we present several numerical examples to illustrate the method. All the numerical work was done with \textsc{matlab}.

\subsection{Analytic boundary}

\begin{example}
\label{ex1}
{\em Union of two disks.}

Here $K$ is the union of two disks of radius $1$ centered at $-2$ and $2$.

\begin{figure}[h!t!b]
\begin{center}
\includegraphics[width=7cm, height=7cm]{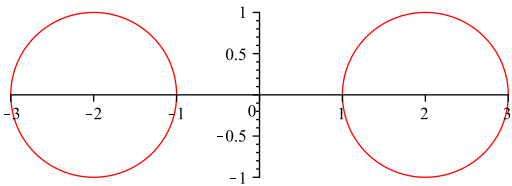}
\caption{The compact set $K$ for Example \ref{ex1}}
\end{center}
\end{figure}

A natural choice here for the approximating functions $g_j$ is to take powers of $1/(z-2)$ and $1/(z+2)$. However, we shall instead consider functions of the form
$$
g_j(z)=\frac{1}{z-a_j},
$$
where the $a_j$'s are distinct points in the interior of $K$. The reason behind this is purely numerical: with these functions, the integrals involved in the method can be calculated analytically, using the residue theorem for example. This way, we avoid the use of numerical quadrature methods, and this results in a significant gain in efficiency.

The locations of the poles $a_j$ are arbitrary. Typically, for each disk centered at $c$ with radius $r$, we put poles at the points
$$
\{c, \,c\pm r_1, \,c\pm r_1i, \,c\pm r_2, \,c\pm r_2i, \dots, c\pm r_n,\,c\pm r_ni\},
$$
where $r_1, \dots, r_n$ are equally distributed between $0$ and $r$.

Table~\ref{table1} contains the bounds for $\gamma(K)$ obtained with the method.

\begin{table}[!hbp]
\begin{center}
\caption{Lower and upper bounds for $\gamma(K)$ for Example \ref{ex1}}
\label{table1}
\begin{tabular}{|c|l|l|r|}
\hline
Poles per disk & Lower bound for $\gamma(K)$ & Upper bound for $\gamma(K)$ & Time (s) \\
\hline
$1$ & 1.875000000000000 &  1.882812500000000 & 0.003279 \\
$5$ & 1.875593064023693 &  1.875619764386366 & 0.007051 \\
$9$ &  1.875595017927203 &  1.875595038756883 & 0.012397 \\
$13$ & 1.875595019096871 &  1.875595019097141 &  0.017422 \\
$17$ & 1.875595019097112 & 1.875595019097164 & 0.027115 \\
\hline
\end{tabular}
\end{center}
\end{table}

We end this example by remarking that, in this particular case, there is a formula for $\gamma(K)$. Indeed, if
$$K= \overline{\mathbb{D}}(-c,r) \cup \overline{\mathbb{D}}(c,r),$$
where $0<r<c$, then we have the formula
\begin{equation}
\label{eqmurai}
\gamma(K)=\frac{r}{2} \Bigl( \frac{1}{\sqrt{q}}-\sqrt{q} \Bigr) \vartheta_2(q)^2.
\end{equation}
Here $\vartheta_2$ is one of the so-called \textit{Jacobi theta-functions}:
\begin{align*}
\label{eqtheta}
\vartheta_2(q)&:= \sum_{n \in \mathbb{Z}} q^{(n+1/2)^2} = 2q^{1/4} \prod_{n=1}^{\infty} (1-q^{2n})(1+q^{2n})^2,\\
\vartheta_3(q)&:= \sum_{n \in \mathbb{Z}} q^{n^2} = \prod_{n=1}^{\infty} (1-q^{2n})(1+q^{2n-1})^2,\\
\vartheta_4(q)&:= \sum_{n \in \mathbb{Z}} (-1)^n q^{n^2} = \prod_{n=1}^{\infty} (1-q^{2n})(1-q^{2n-1})^2.
\end{align*}
The argument $q$ is given by the solution in $(0,1)$ of the equation
\begin{equation*}
\label{eqmurai2}
\frac{c}{r} = \frac{1}{2} \Bigl( \frac{1}{\sqrt{q}}+\sqrt{q} \Bigr).
\end{equation*}
An easy calculation gives
$$
q=\frac{2c^2-r^2-2c\sqrt{c^2-r^2}}{r^2}.
$$
Formula (\ref{eqmurai}) is easily deduced from a formula of Murai in \cite{MUR}, by making the well-known change of variables
$$
k=\frac{\vartheta_2(q)^2}{\vartheta_3(q)^2}
$$
and using the identities relating theta-functions and elliptic integrals.

(We mention though that, in the formula for $\gamma(K)$ in \cite{MUR}, there is a factor $c$ missing, and the formula should read
$$
\gamma(K) = \frac{2}{\pi} ck F(k) \tanh \Bigl(\frac{\pi}{2} \frac{F(\sqrt{1-k^2})}{F(k)} \Bigr),
$$
where $F$ is the complete elliptic integral of the first kind.)

A simple calculation shows that formula (\ref{eqmurai}) can also be written in the form
$$
\gamma(K)= \sqrt{c^2-r^2}\vartheta_2(q)^2.
$$
Substituting $c=2$ and $r=1$ gives
$$
\gamma(K)\approx 1.8755950190971197289.
$$
Compare this with the bounds obtained in Table~\ref{table1}.
\end{example}

\newpage

\begin{example}
\label{ex2}
{\em Union of $25$ disks.}

Each disk in Figure~\ref{F:25disks} has a radius of $0.4$.

\begin{figure}[!h]
\begin{center}
\includegraphics[width=8cm, height=8cm]{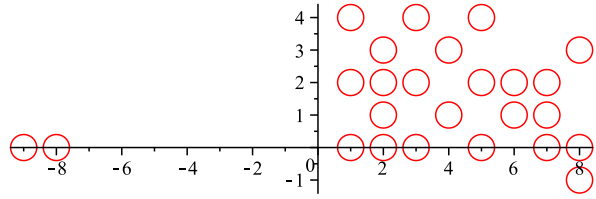}
\caption{The compact set $K$ for Example \ref{ex2}}
\label{F:25disks}
\end{center}
\end{figure}

\begin{table}[!h]
\begin{center}
\caption{Lower and upper bounds for $\gamma(K)$ for Example \ref{ex2}}
\begin{tabular}{|c|l|l|r|}
\hline
Poles per disk & Lower bound for $\gamma(K)$ & Upper bound for $\gamma(K)$ & Time (s) \\
\hline
$1$ &   4.073652478223290 &     4.219704181009330 &    0.177746 \\
$5$ &    4.148169157685863 &    4.148514554979665 &  3.702191 \\
$9$ &   4.148331342401185 &  4.148332498165111 &    11.606526 \\
$13$ &   4.148331931858607 &    4.148331938572625 &    24.848263 \\
$17$ &     4.148331934292544 &   4.148331934334756 &  41.342390 \\
\hline
\end{tabular}
\end{center}
\end{table}

\end{example}

\newpage

\begin{example}
\label{ex3}
{\em Union of four ellipses.}

Here is another example for the computation of the analytic capacity of a compact set with analytic boundary.
The compact set $K$ is composed of four ellipses centered at $-3$, $3$, $10i$, $-10i$. Each ellipse has a semi-major axis of $2$ and a semi-minor axis of $1$:

\begin{figure}[!h]
\begin{center}
\includegraphics[width=8cm, height=8cm]{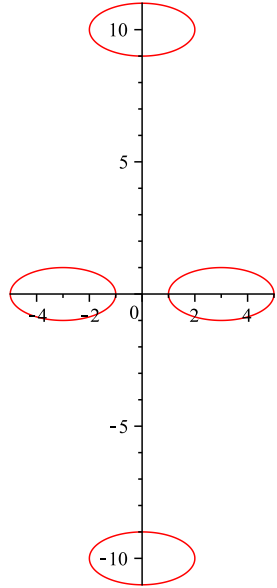}
\caption{The compact set $K$ for Example \ref{ex3}}
\end{center}
\end{figure}

\begin{table}[!h]
\caption{Lower and upper bounds for $\gamma(K)$ for Example \ref{ex3}}
\begin{center}
\begin{tabular}{|c|l|l|r|}
\hline
Poles per ellipse & Lower bound for $\gamma(K)$ & Upper bound for $\gamma(K)$ & Time (s) \\
\hline
$1$ &   4.290494449193028 &      5.652385361295098&     0.962078 \\
$5$ &   5.252560204660928 &      5.409346641724527 &    17.268477 \\
$9$ &   5.356419530523225 &     5.377445892435984 &    54.260216 \\
$13$ &   5.370292494009306 &    5.372648058950175 &  111.424592 \\
$17$ &    5.371877137036634 &     5.372044462730262&   190.042871 \\
$41$ &    5.371995432221965 &   5.371995878776166 & 1100.468881 \\
\hline
\end{tabular}
\end{center}
\end{table}

In this case, the integrals involved have to be calculated numerically. We used a recursive adaptive Simpson quadrature with an absolute error tolerance of $10^{-9}$.
\end{example}

\newpage
\subsection{Piecewise-analytic boundary}
In this subsection, we shall consider examples of compact sets $K$ whose complements $\Omega$ are finitely connected domains with piecewise-analytic boundary.

\begin{example}
\label{exsquare}
{\em The square.}

In this example, we consider the square with corners $1$, $i$, $-1$, $-i$.

\begin{figure}[!h]
\begin{center}
\includegraphics[width=5cm, height=5cm]{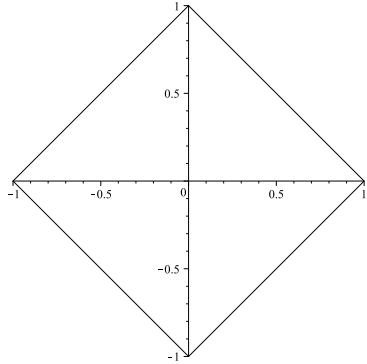}
\caption{The compact set $K$ for Example \ref{exsquare}}
\end{center}
\end{figure}

We fix an integer $n$, and then consider the approximating functions
$$
\frac{1}{z}, \frac{1}{z^2}, \dots, \frac{1}{z^{n}}.
$$
Table~\ref{Ta:square} lists the bounds obtained for different values of $n$:

\begin{table}[!h]
\begin{center}
\caption{Lower and upper bounds for $\gamma(K)$ for Example \ref{exsquare}}
\label{Ta:square}
\begin{tabular}{|c|l|l|r|}
\hline
$n$ & Lower bound for $\gamma(K)$ & Upper bound for $\gamma(K)$ & Time (s) \\
\hline
$2$ &  0.707106781186547 &     0.900316316157106 &     0.021981 \\
$3$ &   0.707106781186547 &   0.900316316157106 & 0.069278\\
$4$ &  0.707106781186547 &   0.887142803070031 &   0.109346\\
$5$ & 0.746499705182962 &   0.887142803070031 &    0.145614 \\
$6$ &   0.746499705182962 & 0.887142803070031 &  0.202309\\
$7$ &   0.746499705182962 &     0.887142803070031 &     0.295450 \\
$8$ &   0.746499705182962 &  0.881014562149127 & 0.347996\\
$9$ &   0.761941423753061 &    0.881014562149127 &    0.414684\\
$10$ &  0.761941423753061 &  0.881014562149127 &     0.595552 \\
$15$ &  0.770723484232218 &  0.877175902241141 &      2.425285 \\
$20$ &  0.776589045256849 &  0.872341829081944 &     5.537981 \\
$25$ &  0.784189460107018 &   0.870656623669828 &      10.002786 \\
$30$ &  0.786857803378602 &   0.869257904380382 &      16.344379 \\
$35$ &  0.789068961951613 &    0.868068649269412 &      26.109797 \\
$40$ &   0.790942498354322 &    0.866133165258689 &      33.595790 \\
\hline
\end{tabular}
\end{center}
\end{table}

\end{example}
We immediately see that the convergence is very slow, as opposed to the results obtained in the case of compact sets with analytic boundaries.
The main issue here is that we do not consider the geometric nature of the boundary. In order to accelerate convergence, our choice of approximating functions should take into account the different points where the boundary fails to be smooth.

In view of Theorems \ref{propp2} and \ref{est}, the functions that we want to approximate are
$$
q(z)=c \sqrt{(\tilde{\psi} \circ F)(z)} \sqrt{F'(z)}
$$
and
$$
f(z)q(z)=c f(z) \sqrt{(\tilde{\psi} \circ F)(z)} \sqrt{F'(z)}
$$
for some constant $c$, where $F$ is a conformal map of $\Omega$ onto $\mathbb{C}_{\infty} \setminus \overline{\mathbb{D}}$, with $F(\infty)=\infty$.
By Theorems \ref{Theo1} and \ref{Theo2}, we know that the functions $f$ and $\sqrt{(\tilde{\psi} \circ F)}$ are continuous up to the boundary. These functions can thus be approximated by rational functions with poles inside the square.
All that remains is to add approximating functions that behave like $\sqrt{F'}$ at the corners. If $a$ is one of the corner in the boundary, then $F$ should, in some sense, straighten out the angle from $3 \pi /2$ to $\pi$, that is $F(z)$ must behave like $$(z-a)^{2/3}$$ near $a$. Differentiating and then taking square root, we find that $(z-a)^{-1/6}$ should be, up to a multiplicative constant, a good approximation to $\sqrt{F'(z)}$ near $a$. Since we want functions that are holomorphic near $\infty$, we shall instead consider
$$
\Bigl( \frac{z-a}{z} \Bigr)^{-1/6}.
$$
In view of all of the above, we propose the following method for the computation of $\gamma(K)$:

Fix an integer $n$. Then add the approximating functions
$$
\frac{f_j(z)}{z^k}
$$
for $j=0,1,2,3,4$ and $k=1,2, \dots, n$, where $a_1,a_2,a_3,a_4$ are the corners of the square,
$$
f_0(z):=1
$$
and
$$
f_j(z):= \Bigl( \frac{z-a_j}{z} \Bigr)^{-1/6} \qquad (j=1,2,3,4).
$$

We use this method to recompute the analytic capacity of the square of Example \ref{exsquare}.
The convergence is significantly faster.

\begin{table}[!h]
\label{tab1}
\begin{center}
\caption{Lower and upper bounds for $\gamma(K)$}
\begin{tabular}{|c|l|l|r|}
\hline
$n$  & Lower bound for $\gamma(K)$ & Upper bound for $\gamma(K)$ & Time (s) \\
\hline
$2$  & 0.834566926465074 &    0.835066810881929 &     1.334885 \\
$3$  &   0.834609482283050 & 0.834678782816948 & 2.918624\\
$4$  &  0.834622127643984 &  0.834628966618492 &  5.220941\\
$5$  & 0.834626255962448 &  0.834627566559480 &   8.022274 \\
$6$ &  0.834626584020641 & 0.834627152182154 &  11.542859\\
\hline
\end{tabular}
\end{center}
\end{table}

We remark that, in this case, the answer can be calculated exactly.
Indeed, since $K$ is connected, we have that
$$
\gamma(K)
= \operatorname{cap}(K)
= \sqrt{2} \frac{\Gamma(1/4)^2}{4\pi^{3/2}} \approx 0.83462684167407318630,
$$
where $\operatorname{cap}(K)$ is the logarithmic capacity of $K$.

Our method can easily be adapted to other compact sets with piecewise-analytic boundary. Indeed, suppose that $K$ is a compact set whose boundary consists of $m$ piecewise-analytic curves, say $\gamma_1, \dots, \gamma_m$. First, fix a point $c$ in the interior of $\gamma_1$, and let $a_1, a_2, \dots, a_N$ be the different points in $\gamma_1$ where the curve fails to be smooth. Suppose that $\gamma_1$ makes an exterior angle of $\alpha_j$ at the point $a_j$, where $0< \alpha_j < 2\pi$. Then we add the following approximating functions:
$$\frac{f_j(z)}{(z-c)^k}$$
for $j=0,1,2, \dots, N$ and $k=1,2, \dots, n$, where
$$
f_0(z):=1
$$
and
$$
f_j(z):= \Bigl( \frac{z-a_j}{z-c} \Bigr)^{(1/2)(\pi/\alpha_j -1)} \qquad (j=1,2, \dots, N).
$$
All that remains is to repeat the procedure for the other curves.

Here is an illustrative example.

\newpage

\begin{example}
\label{ex4}
{\em Union of two squares, one equilateral triangle and one rectangle.}

\begin{figure}[!h]
\begin{center}
\includegraphics[width=7cm, height=7cm]{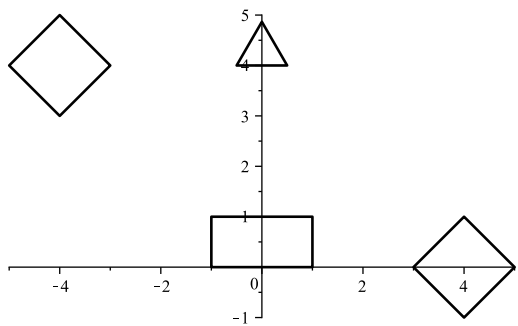}
\caption{The compact set $K$ for Example \ref{ex4} }
\end{center}
\end{figure}

\begin{table}[!h]
\begin{center}
\caption{Lower and upper bounds for $\gamma(K)$ for Example \ref{ex4}}
\begin{tabular}{|c|l|l|r|}
\hline
$n$  & Lower bound for $\gamma(K)$ & Upper bound for $\gamma(K)$ & Time (s) \\
\hline
$1$ &  2.688593215018632 &     2.724269900679792 &       10.371944 \\
$2$ &     2.693483826380926 &    2.695819902453329 &  32.881242\\
$3$ &  2.693867645864377&    2.694261483861710 &   71.562216\\
$4$ &   2.693961062687599 &   2.694025016036611 &    122.607285 \\
$5$ &     2.693971930182724 &   2.693982653270314&   184.203252\\
\hline
\end{tabular}
\end{center}
\end{table}
\end{example}

\newpage

\begin{example}
\label{ex5}
{\em Union of a disk and two semi-disks}

Our last example is a non-polygonal compact set with piecewise-analytic boundary. It is a typical example of the kind of geometry that often arises in applied mathematics, featuring smoothness of the boundary with the exception of a few singularities.

\begin{figure}[!h]
\begin{center}
\includegraphics[width=7cm, height=7cm]{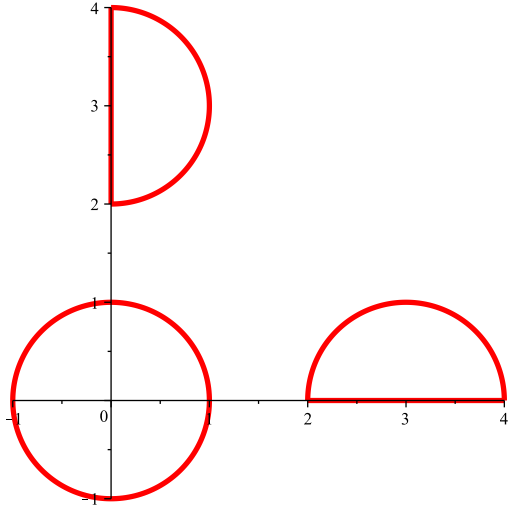}
\caption{The compact set $K$ for Example \ref{ex5}}
\end{center}
\end{figure}

The compact set $K$ is composed of the unit disk and two half-unit-disks centered at $3$ and $3i$.

\begin{table}[h]
\begin{center}
\caption{Lower and upper bounds for $\gamma(K)$ for Example \ref{ex5}}
\begin{tabular}{|c|l|l|r|}
\hline
$n$ & Lower bound for $\gamma(K)$ & Upper bound for $\gamma(K)$ & Time (s) \\
\hline
$2$ & 2.118603690751346 &    2.123888275897654 &     2.546965 \\
$3$  &    2.120521869940459 & 2.121230615594293 &  4.926440\\
$4$  &   2.120666182274863 &  2.120803766391281 &  9.488024 \\
$5$  &  2.120694837101383 &   2.120716977856280 &   13.679742 \\
$6$  &   2.120703235395670 &  2.120709388805280&  22.344576\\
$7$  &   2.120704581010457 & 2.120707633546616 &  28.953791\\
$8$  &   2.120705081159854 &  2.120706704970516&  34.781046\\
\hline
\end{tabular}
\end{center}
\end{table}

The above compact set was considered in \cite{ROS} and then in \cite{RANS}, for the computation of logarithmic capacity. It was shown that $\operatorname{cap}(K) \in [2.19699, 2.19881]$. Our results are thus consistent with the well-known inequality
$$
\gamma(K) \leq \operatorname{cap}(K).
$$
\end{example}

Before leaving this section, a brief remark is in order. Comparing the results of Subsection $6.2$ with Subsection $6.1$, we see that the convergence is quite a bit slower in the case of piecewise-analytic boundary, compared to the case of analytic boundary. In fact, it is known that one \textit{cannot} hope for similar convergence in both cases. This is related to the fact that if the boundary curves are piecewise-analytic but not analytic, then the extremal functions in Theorem \ref{est} do not extend analytically across the boundary.

\section{The subadditivity problem for analytic capacity}
\label{sec7}
This section is about the study of the following question:
\emph{Is it true that
\begin{equation}
\label{eqsub}
\gamma(E \cup F) \leq \gamma(E) + \gamma(F)
\end{equation}
for all compact sets $E,F$?}

Suita \cite{SUI} proved that (\ref{eqsub}) holds if $E,F$ are disjoint \textit{connected} compact sets.

One of the main obstacles in the study of inequality (\ref{eqsub}) is that it is difficult in practice to determine the analytic capacity of a given compact set. However, the numerical examples in the last section show that our method is very efficient when the compact in question is a finite union of disjoint disks. Fortunately, this particular case is sufficient:

\begin{theorem}
\label{theoo}
The following are equivalent:
\begin{enumerate}[\rm(i)]
\item  $\gamma(E \cup F) \leq \gamma(E) + \gamma(F)$ for all compact sets $E, F$.
\item $\gamma(E \cup F) \leq \gamma(E) + \gamma(F)$ for all disjoint compact sets $E, F$ that are finite unions of disjoint closed disks, all with the same radius.
\end{enumerate}
\end{theorem}

Clearly (i) implies (ii), but the fact that the converse holds is nontrivial. For the proof of Theorem \ref{theoo}, we need a discrete approach to analytic capacity introduced by Melnikov \cite{MEL}.

\subsection{Melnikov's discrete approach to Analytic Capacity}

Let $z_1, z_2, \dots, z_n \in \mathbb{C}$ and let $r_1, r_2, \dots, r_n$ be positive real numbers. Define $Z:= (z_1, z_2, \dots, z_n)$ and $R:=(r_1, r_2, \dots, r_n)$. Suppose in addition that $|z_j-z_k|>r_j+r_k$ for $j \neq k$, so that the closed disks $\overline{\mathbb{D}}(z_j,r_j)$ are pairwise disjoint. Set
$$
K(Z,R):=\bigcup_{j=1}^{n}\overline{\mathbb{D}}(z_j,r_j)
$$
and let
$$
\mu_1=\mu_1(Z,R):=\sup \Bigl\{ \Bigl| \sum_{j=1}^n a_j \Bigr| \Bigr\}
$$
where the supremum is taken over all points $a_1, \dots, a_n \in \mathbb{C}$ such that
$$
\Bigl| \sum_{j=1}^n \frac{a_j}{z-z_j} \Bigr| \leq 1 \qquad (z \in \mathbb{C} \setminus K(Z,R)).
$$
Clearly, we have $\mu_1 \leq \gamma(K(Z,R))$.

Finally, for any compact set $K \subseteq \mathbb{C}$ and $\delta>0$, we write $K_{\delta}$ for the closed $\delta-$neighborhood of $K$.

The following lemma is precisely what we need to prove Theorem \ref{theoo}:
\begin{lemma}
\label{meln}
Let $K \subseteq \mathbb{C}$ compact, and let $\delta, \epsilon > 0$. Then there exist $z_1, \dots, z_n \in K_\delta$ and $0<r<\delta$ such that $|z_j - z_k| > 2r$ for $j \neq k$, and
$$
\mu_1(Z,R) \geq (1-\epsilon) \gamma(K)
$$
where $Z=(z_1, \dots, z_n)$ and $R=(r, \dots, r).$ In particular,
$$
\gamma(K(Z,R)) \geq (1-\epsilon)\gamma(K).
$$
\end{lemma}

\begin{proof}
See \cite[Lemma 1]{MEL}.
\end{proof}

We can now prove the implication (ii)$\Rightarrow$(i) in Theorem \ref{theoo}:

\begin{proof}
Suppose that  (i) does not hold, so there exist compact sets $E,F$ with
$$
\gamma(E \cup F) > \gamma(E) + \gamma(F).
$$
Let $0<\epsilon < \gamma(E \cup F) - \gamma(E) - \gamma(F)$. Take $\delta >0$ sufficiently small so that
\begin{equation}
\label{eq1}
\gamma(E_{2\delta}) < \gamma(E) + \epsilon/3,
\end{equation}
and
\begin{equation}
\label{eq2}
\gamma(F_{2\delta}) < \gamma(F) + \epsilon/3.
\end{equation}
By Lemma \ref{meln}, there exist $z_1, z_2, \dots, z_n \in (E \cup F)_{\delta}$ and $0<r<\delta$ such that
$$
\gamma(K(Z,R)) \geq \gamma(E \cup F) - \epsilon/3
$$
and
the disks $\overline{\mathbb{D}}(z_j,r)$ are pairwise disjoint. For each $j \in \{1,2, \dots, n\}$, fix $w_j \in E \cup F$ with $|z_j-w_j|=\operatorname{dist}(z_j,E \cup F) \leq \delta$. Let $A$ be the union of the disks $\overline{\mathbb{D}}(z_j,r)$ with $w_j \in E$, and let $B$ be the union of the disks $\overline{\mathbb{D}}(z_k,r)$ with $w_k \in F \setminus E$.
Then $A \subseteq E_{2\delta}$ and $B \subseteq F_{2\delta}$. Since
$$
\gamma(A \cup B) \geq \gamma(E \cup F) - \epsilon/3,
$$
we have
\begin{align*}
\gamma(A \cup B) &\geq \gamma(E \cup F) - \epsilon/3\\
&> \gamma(E) + \gamma(F) + \epsilon - \epsilon/3\\
&= \gamma(E) + \epsilon/3 + \gamma(F) + \epsilon/3\\
&> \gamma(A) + \gamma(B),
\end{align*}
where we used equations (\ref{eq1}) and (\ref{eq2}). Therefore (ii) fails to hold.
\end{proof}

\subsection{Discrete Analytic Capacity}

For $Z=(z_1, \dots, z_n) \in \mathbb{C}^n$, where $z_j \neq z_k $ for all $j \neq k$, and $r>0$, define
$$
\gamma(Z,r):=\gamma\Bigl(\bigcup_{j=1}^{n}\overline{\mathbb{D}}(z_j,r)\Bigr).
$$
Assume in addition that the discs $\overline{\mathbb{D}}(z_j,r)$ are pairwise disjoint.
By Theorem \ref{theoo}, the subadditivity of analytic capacity is equivalent to
$$
\gamma(Z,r) \leq \gamma(Z',r) + \gamma(Z'',r)
$$
for all $z_1, \dots, z_n \in \mathbb{C}$
and all $m \in \{1,2, \dots, n-1\}$,
where $Z=(z_1, \dots, z_n)$, $Z'=(z_1, \dots, z_m)$ and $Z'' = (z_{m+1}, \dots, z_n)$.
The above inequality can be written as
$$
R(Z,r,m) \leq 1,
$$
where
$$
R(Z,r,m):= \frac{\gamma(Z,r)}{\gamma(Z',r) + \gamma(Z'',r)}.
$$

The above shows the importance of studying the quantity $R(Z,r,m)$.
In this subsection, our objective is to obtain the following asymptotic expression for $R(Z,r,m)$:

\begin{theorem}
\label{theoratio}
Fix $z_1, z_2, \dots, z_n \in \mathbb{C}$ and fix $m \in \{1,2, \dots, n-1\}$. Then
$$
R(Z,r,m)=1-Cr^2+O(r^3)
$$
as $r \rightarrow 0$, where $C$ is a strictly positive constant depending only on $m,n$ and $z_1, z_2, \dots, z_n$.
\end{theorem}

For the proof of Theorem \ref{theoratio}, we need to introduce a discrete version of analytic capacity, first considered by Melnikov \cite{MEL}.

For $Z=(z_1, \dots, z_n)$ and $r>0$, define the \textit{discrete analytic capacity} $\lambda$ by
$$
\lambda(Z,r):= \sup \Bigl\{ \Bigl| \sum_{j=1}^n a_j \Bigr| ^2\Bigr\},
$$
where the supremum is taken over all $(a_1, \dots, a_n) \in \mathbb{C}^n$ such that
$$
\sum_{k=1}^n \Bigl( \frac{|a_k|^2}{r} + r \Bigl| \sum_{j \neq k} \frac{a_j}{z_k-z_j} \Bigr|^2 \Bigr) \leq 1.
$$

We also introduce the following constants:
\begin{align*}
M(Z,r)&:= r^4\sum_k \sum_{j \neq k} \frac{1}{|z_k-z_j|^4} \\
N(Z,r)&:=r\Bigl( \sum_k \sum_{j \neq k} \frac{1}{|z_k-z_j|^2} \Bigr) ^{1/2} M(Z,R)^{1/2}.
\end{align*}

Then we have

\begin{theorem}
\label{theo2}
Let $Z=(z_1, \dots, z_n) \in \mathbb{C}^n$ and $r>0$. Suppose that the discs $\overline{\mathbb{D}}(z_j, 2r)$ are pairwise disjoint. Then
$$
\frac{\gamma(Z,r)}{1+4N(Z,r)} \leq \lambda(Z,r) \leq (1+2M(Z,r))\gamma(Z,r).
$$
\end{theorem}

\begin{corollary}
Let $K \subseteq \mathbb{C}$ compact, and let $\delta > 0$ and $\epsilon>0$. Then there exists $z_1, \dots, z_n \in K_\delta$ and $0<r<\delta$ such that $|z_j - z_k| > 2r$ for $j \neq k$ and
$$
|\gamma(K) - \gamma(Z,r)|<\epsilon,
$$
where $Z=(z_1, \dots, z_n)$. Furthermore, $Z$ and $r$ can be chosen so that $M(Z,r)<\epsilon$ and $N(Z,r) < \epsilon$.
\end{corollary}

\begin{proof}[Proofs]
See \cite[Theorem 2 and Corollary]{MEL}.
\end{proof}

Now, we need another expression for $\lambda(Z,r)$ which is easier to manipulate. We proceed as in \cite{MEL}.

It is not hard to show that
$$
\sum_{k=1}^n \Bigl( \frac{|a_k|^2}{r} + r \Bigl| \sum_{j \neq k} \frac{a_j}{z_k-z_j} \Bigr|^2 \Bigr)
= \langle (D_R^{-1}+B)a,a \rangle,
$$
where $a=(a_1, \dots, a_n)$, $D_R$ is the diagonal matrix with each entry of the main diagonal equal to $r$, and $B=(b_{jk})$, where
$$
b_{jk} = \sum_{m\neq j,k} \frac{r}{(z_j-z_m)(\overline{z_k-z_m})}.
$$
Here $\langle \cdot , \cdot \rangle$ is the standard scalar product in $\mathbb{C}^n$.
The matrix $B$ can be written in the form $B=C D_R C^{*}$, where $C=(c_{jk})$ is the Cauchy matrix associated with $z_1, \dots, z_n$, i.e.
$$
c_{jk} = \frac{1}{z_j-z_k} \qquad (j \neq k), \qquad c_{jj}=0.
$$
Arguing as in \cite[Lemma 3]{MEL}, we obtain
$$
\lambda(Z,r) = \langle (D_R^{-1}+C D_R C^{*})^{-1}(\textbf{1}),\textbf{1} \rangle,
$$
where $\textbf{1} = (1, 1, \dots, 1) \in \mathbb{C}^n.$

The following lemma contains estimates for the discrete analytic capacity:

\begin{lemma}
\label{lem2}
Let $Z=(z_1, \dots, z_n) \in \mathbb{C}^n$ and let $r>0$. Then
$$
nr -r^3 \langle CC^{*}(\textbf{1}),\textbf{1} \rangle \leq \lambda(Z,r)
\leq nr-r^3 \langle CC^{*}(\textbf{1}),\textbf{1} \rangle + r^5 \langle CC^{*}CC^{*}(\textbf{1}),\textbf{1} \rangle.
$$
\end{lemma}

\begin{proof}
See \cite[Lemma 4]{MEL}.
\end{proof}

We shall also need the following lemma:

\begin{lemma}
\label{lemratio}
Fix $z_1, \dots, z_n \in \mathbb{C}$ and, as before, let $Z=(z_1, \dots, z_n)$, $Z'=(z_1, \dots, z_m)$ and $Z''=(z_{m+1},\dots,z_n)$.
Write $\alpha:= \langle CC^{*}(\textbf{1}),\textbf{1}\rangle$, $\alpha':=\langle C'C'^{*}(\textbf{1}),\textbf{1}\rangle$ and $\alpha'':=\langle C''C''^{*}(\textbf{1}),\textbf{1}\rangle$, where $C'$ is the Cauchy matrix associated with $Z'$, and $C''$ the Cauchy matrix associated with $Z''$. Then
$$
\alpha> \alpha' + \alpha''.
$$
\end{lemma}

\begin{proof}
First note that we have the following expression for $\alpha$:
\begin{align*}
\langle CC^{*}(\textbf{1}),\textbf{1}\rangle &= \langle C^{*}(\textbf{1}), C^{*}(\textbf{1}) \rangle \\
&= \sum_j \sum_k \sum_{l \neq j,k} \frac{1}{(z_k-z_l)(\overline{z_j-z_l})} \\
&= \sum_j \sum_{l \neq j} \frac{1}{|z_j-z_l|^2} + \sum_j \sum_{k \neq j} \sum_{l \neq j,k} \frac{1}{(z_k-z_l)(\overline{z_j-z_l})}.
\end{align*}
An elementary calculation shows that the second sum is equal to
$$
\sum_{j<k<l}\Bigl(\frac{4S(z_j,z_k,z_l)}{|z_j-z_k||z_j-z_l||z_k-z_l|} \Bigr)^2
= \sum_{j<k<l}\frac{1}{R(z_j,z_k,z_l)^2},
$$
where $S(z_j,z_k,z_l)$ is the area of the triangle with vertices $z_j, z_k, z_l$
and $R(z_j,z_k,z_l)$ is the radius of the circle through $z_j, z_k,z_l$
(if $z_j, z_k, z_l$ are collinear, then we set $S(z_j,z_k,z_l):=0$ and $R(z_j,z_k,z_l):=\infty$).
The conclusion follows.
\end{proof}

We can now proceed to the proof of Theorem \ref{theoratio}:

\begin{proof}
Write $M(Z,r):=Ar^4$ and $N(Z,r):=Br^3$, where $A,B$ do not depend on $r$.
Also, define
$$
\beta:=\langle CC^{*}CC^{*}(\textbf{1}),\textbf{1} \rangle,
\quad \beta':=\langle C'C'^{*}C'C'^{*}(\textbf{1}),\textbf{1} \rangle,
\quad \beta'':=\langle C''C''^{*}C''C''^{*}(\textbf{1}),\textbf{1} \rangle,
$$
where $C, C', C''$ are as in Lemma \ref{lemratio}.

By Theorem \ref{theo2} and Lemma \ref{lem2}, we have
\begin{align*}
R(Z,r,m) &\leq \frac{(1+4Br^3)(1+2Ar^4)\lambda(Z,r)}{\lambda(Z',r)+\lambda(Z'',r)}\\
&\leq (1+4Br^3)(1+2Ar^4) \frac{nr - \alpha r^3 + \beta r^5}{mr-\alpha' r^3 + (n-m)r - \alpha'' r^3} \\
&= (1+4Br^3)(1+2Ar^4) \frac{n-\alpha r^2 + \beta r^4}{n - (\alpha' + \alpha'') r^2}.
\end{align*}
Now, by Lemma \ref{lemratio}, $\alpha=\alpha'+\alpha''+\delta$ for some $\delta>0$, and thus
\begin{align*}
R(Z,r,m) &\leq (1+4Br^3)(1+2Ar^4) \frac{n-(\alpha' + \alpha'' + \delta) r^2 + \beta r^4}{n - (\alpha' + \alpha'') r^2}\\
&= (1+4Br^3)(1+2Ar^4) \Bigl( 1 - r^2 \frac{\delta }{n - (\alpha' + \alpha'') r^2} + r^4 \frac{\beta}{n - (\alpha' + \alpha'') r^2} \Bigr)\\
&= (1+4Br^3)(1+2Ar^4) \Bigl( 1 - \frac{\delta}{n} r^2 + O(r^4) \Bigr)\\
&= 1 - \frac{\delta}{n} r^2 + O(r^3),
\end{align*}
as $r \rightarrow 0$.

Proceeding similarly, we obtain the reverse inequality:
\begin{align*}
R(Z,r,m) &\geq (1+2Ar^4)^{-1}(1+4Br^3)^{-1} \frac{\lambda(Z,r)}{\lambda(Z',r)+\lambda(Z'',r)}\\
&\geq (1+2Ar^4)^{-1}(1+4Br^3)^{-1} \frac{nr-\alpha r^3}{mr-\alpha' r^3 + \beta' r^5 + (n-m)r - \alpha'' r^3 + \beta'' r^5}\\
&= (1+2Ar^4)^{-1}(1+4Br^3)^{-1} \frac{n-\alpha r^2}{n - (\alpha'+\alpha'')r^2 + (\beta'+\beta'')r^4}\\
&= (1+2Ar^4)^{-1}(1+4Br^3)^{-1} \frac{n-(\alpha'+\alpha'') r^2 - \delta r^2}{n - (\alpha'+\alpha'')r^2 + (\beta'+\beta'')r^4}\\
&= (1+2Ar^4)^{-1}(1+4Br^3)^{-1} \Bigl( 1 - \frac{\delta r^2 + (\beta'+\beta'')r^4 }{n - (\alpha'+\alpha'')r^2 + (\beta'+\beta'')r^4} \Bigr)\\
&= (1+2Ar^4)^{-1}(1+4Br^3)^{-1} \Bigl( 1 - \frac{\delta}{n} r^2 + O(r^4) \Bigr)\\
&= 1 - \frac{\delta}{n} r^2 + O(r^3)
\end{align*}
as $r \rightarrow 0$.
\end{proof}

\begin{remark}
One immediately deduces from Theorem \ref{theoratio} some interesting properties of the ratio $R(Z,r,m)$. First of all, for fixed $Z$ and $m$, the asymptotic expression implies that $R(Z,r,m) \leq 1$ for all $r$ sufficiently small. Second, we have that $R(Z,r,m) \rightarrow 1$ as $r \rightarrow 0$. Another interesting consequence of Theorem \ref{theoratio} will be discussed in the next section.
\end{remark}

\section{A conjecture related to the subadditivity problem}
\label{sec8}

\subsection{Formulation of the conjecture}
Recall that, by Theorem \ref{theoo}, the subadditivity of analytic capacity is equivalent to the subadditivity in the case of disjoint finite unions of disjoint disks, all with the same radius. For such compact sets, our method for the computation of analytic capacity is very efficient, see e.g. the numerical examples of Section \ref{sec6}. This allowed us to perform a lot of numerical experiments.

More precisely, let $Z=(z_1, \dots, z_n) \in \mathbb{C}^n$ and $m \in \{1, \dots, n-1\}$. Let $r>0$,
chosen sufficiently small so that the disks $\overline{\mathbb{D}}(z_j,r)$ are pairwise disjoint. Recall that $R(Z,r,m)$ is defined by
$$
R(Z,r,m)=\frac{\gamma(E \cup F)}{\gamma(E) + \gamma(F)},
$$
where
$$
E:=\bigcup_{j=1}^m\overline{\mathbb{D}}(z_j,r)
$$
and
$$
F:=\bigcup_{j=m+1}^n\overline{\mathbb{D}}(z_j,r).
$$
Using our numerical method, we can easily compute upper and lowers bounds for the ratio $R(Z,r,m)$.

All the numerical experiments that we have performed seem to suggest that analytic capacity is indeed subadditive, i.e.
that
$$
R(Z,r,m) \leq 1
$$
for all $Z,r,m$. More surprising though, all these experiments seem to indicate that the ratio $R(Z,r,m)$ \textit{decreases} as $r$ \textit{increases}. We formulate this as a conjecture:

\begin{conjecture}
Fix $Z=(z_1, \dots, z_n)$ and $m \in \{1, \dots, n-1\}$. Then $R(Z,r,m)$ is a decreasing function of $r$.
\end{conjecture}

In view of Theorem \ref{theoratio}, a proof of the above conjecture would imply that analytic capacity is subadditive.
Moreover, Theorem \ref{theoratio} also implies that, for fixed $Z$ and $m$, the above holds for all $r$ sufficiently small.
We now present several numerical experiments to illustrate the conjecture.

\newpage

\subsection{Numerical examples}

\begin{example}
\label{ex6}
For the first example, the number of disks is $n=40$, and $m=20$. The compact set $E$ is composed of the $20$ disks with bold boundaries, and $F$ is the union of the remaining disks.

\begin{figure}[!h]
\begin{center}
\includegraphics[width=6cm, height=6cm]{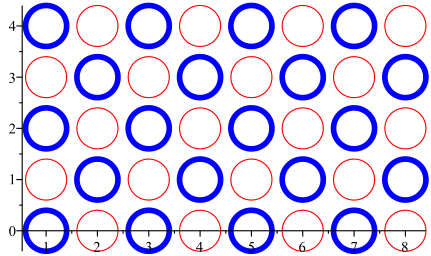}
\caption{The set $E \cup F$ for Example \ref{ex6}}
\end{center}
\end{figure}

For $500$ values of the radius $r$ equally distributed between $0$ and $0.499$, we computed lower and upper bounds for the ratio $R(z,r,m)$.
Figure~\ref{F:40disksgraph} shows the graph of the lower bound versus $r$. The graph for the upper bound is almost identical; the two graphs differ by at most $0.002481$.

\begin{figure}[!h]
\begin{center}
\includegraphics[width=7cm, height=7cm]{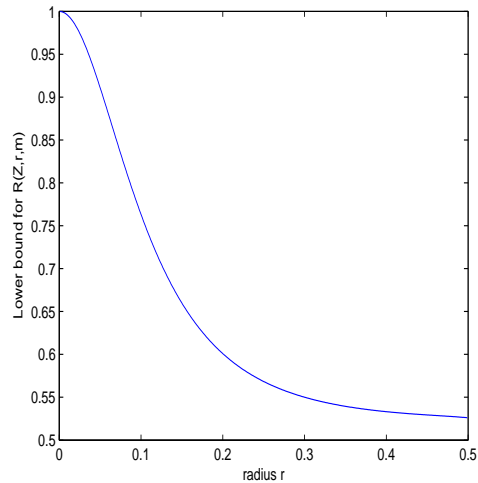}
\caption{Graph of the ratio $\gamma(E \cup F)/(\gamma(E)+\gamma(F))$ for Example \ref{ex6}}
\label{F:40disksgraph}
\end{center}
\end{figure}

\end{example}

\begin{example}
\label{ex9}
Here the disks are centered at equally spaced points in the real axis. The number of disks is $n=10$, and $m=5$.

\begin{figure}[!h]
\begin{center}
\includegraphics[width=7cm, height=7cm]{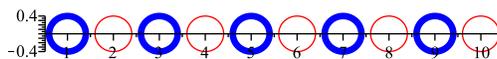}
\caption{The set $E \cup F$ for Example \ref{ex9}}
\end{center}
\end{figure}

\begin{figure}[!h]
\begin{center}
\includegraphics[width=7cm, height=7cm]{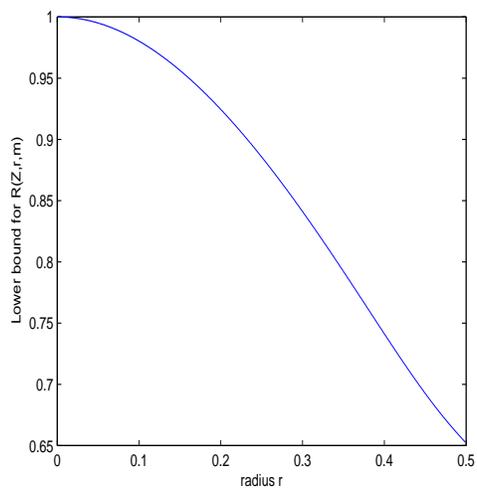}
\caption{Graph of the ratio $\gamma(E \cup F)/(\gamma(E)+\gamma(F))$ for Example \ref{ex9}}
\end{center}
\end{figure}

\end{example}

\newpage

\begin{example}
\label{ex7}
Here the disks are centered randomly. The number of disks is $n=18$, and $m=12$.

\begin{figure}[!h]
\begin{center}
\includegraphics[width=7cm, height=7cm]{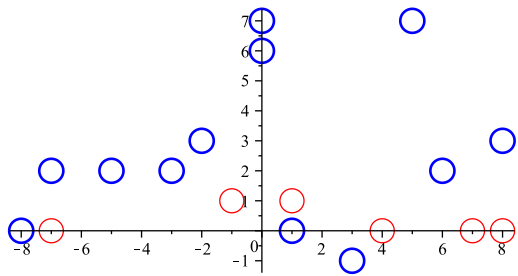}
\caption{The set $E \cup F$ for Example \ref{ex7}}
\end{center}
\end{figure}

\begin{figure}[!h]
\begin{center}
\includegraphics[width=7cm, height=7cm]{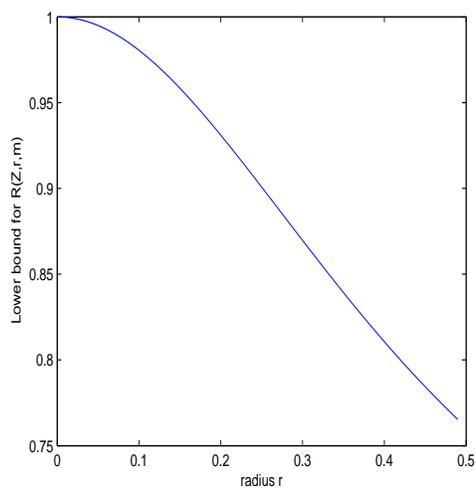}
\caption{Graph of the ratio $\gamma(E \cup F)/(\gamma(E)+\gamma(F))$ for Example \ref{ex7}}
\end{center}
\end{figure}

\end{example}

\newpage

\begin{example}
\label{ex8}
This last example shows that the situation can be different if, instead, we fix the radius of some of the disks and let the radius of the other disks vary.
Here the disks on the left are fixed, with radius $0.49$. The radius of the two disks centered at $10$ and $11$ vary simultaneously, from $0$ to $0.499$.

\begin{figure}[h!t!b]
\begin{center}
\includegraphics[width=7cm, height=7cm]{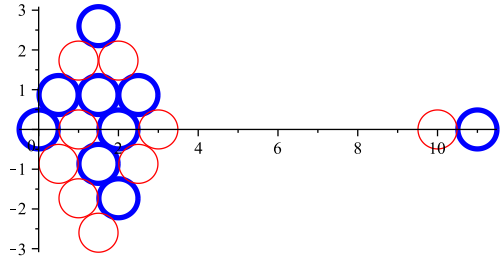}
\caption{The set $E \cup F$ for Example \ref{ex8}}
\end{center}
\end{figure}

\begin{figure}[h!t!b]
\begin{center}
\includegraphics[width=7cm, height=7cm]{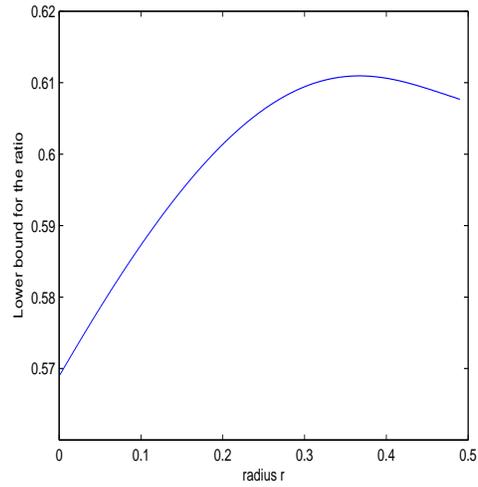}
\caption{Graph of the ratio $\gamma(E \cup F)/(\gamma(E)+\gamma(F))$ for Example \ref{ex8}}
\end{center}
\end{figure}

\end{example}

\newpage

\subsection{Proof of the conjecture in the case $n=2$}
We end this section by giving a proof of the conjecture in the simplest case.

\begin{theorem}
\label{theodec}
Let $E$ and $F$ be disjoint closed disks of radius $r$. Then
$$\frac{\gamma(E \cup F)}{\gamma(E)+\gamma(F)}$$
is a decreasing function of $r$.
\end{theorem}

\begin{proof}
The main ideas of the proof that follows were suggested to us by Juan Arias de Reyna, and we gratefully acknowledge his contribution.

Without loss of generality, we can suppose that $E$ and $F$ are centered at $c$ and $-c$ respectively, where $c>0$ and $0<r<c$. In this case, we have, by formula (\ref{eqmurai}) of Example \ref{ex1},
$$
\gamma(E \cup F)
=\frac{r}{2} \Bigl( \frac{1}{\sqrt{q}}-\sqrt{q} \Bigr) \vartheta_2(q)^2,
$$
where $q$ is the solution in $(0,1)$ of the equation
\begin{equation}
\label{eqq2}
\frac{c}{r} = \frac{1}{2} \Bigl( \frac{1}{\sqrt{q}}+\sqrt{q} \Bigr).
\end{equation}
Recall that the Jacobi theta-functions are defined by
\begin{align*}
\label{eqtheta}
\vartheta_2(q)&:= \sum_{n \in \mathbb{Z}} q^{(n+1/2)^2} = 2q^{1/4} \prod_{n=1}^{\infty} (1-q^{2n})(1+q^{2n})^2,\\
\vartheta_3(q)&:= \sum_{n \in \mathbb{Z}} q^{n^2} = \prod_{n=1}^{\infty} (1-q^{2n})(1+q^{2n-1})^2,\\
\vartheta_4(q)&:= \sum_{n \in \mathbb{Z}} (-1)^n q^{n^2} = \prod_{n=1}^{\infty} (1-q^{2n})(1-q^{2n-1})^2,
\end{align*}
for $q \in (0,1)$.

Now, since $\gamma(E)=\gamma(F)=r$, we have that
$$
\frac{\gamma(E \cup F)}{\gamma(E) + \gamma(F)}
= \frac{1}{4} \Bigl( \frac{1}{\sqrt{q}}-\sqrt{q} \Bigr) \vartheta_2(q)^2:=f(q).
$$
By equation (\ref{eqq2}), $q$ increases from $0$ to $1$ as $r$ increases from $0$ to $c$. It thus suffices to show that $f(q)$, defined above, is a decreasing function of $q \in (0,1)$. Proving this directly seems difficult, mainly because of the difference in the behavior of $f$ near $1$, and away from $1$. For this reason, we separate the proof in two cases:
\bigskip

\textbf{Case 1: $q \in (0, 0.8]$}

The idea in this case is to express $f$ as a Jacobi product and then compute the logarithmic derivative. Using the product expression for $\vartheta_2(q)$, we obtain
\begin{equation*}
f(q)= (1-q) \prod_{n=1}^{\infty} (1-q^{2n})^2 (1+q^{2n})^4 =(1-q) \prod_{n=1}^{\infty} (1-q^{4n})^2 (1+q^{2n})^2
\end{equation*}
Since $f$ is positive on $(0,1)$, proving that $f$ is decreasing in this interval is equivalent to proving that its logarithmic derivative is negative. Computing the logarithmic derivative and multiplying by $q$, we obtain the function
$$
u(q):= q \frac{f'(q)}{f(q)}=-\frac{q}{1-q} - \sum_{n=1}^{\infty} \frac{8nq^{4n}}{1-q^{4n}} + \sum_{n=1}^{\infty}\frac{4nq^{2n}}{1+q^{2n}}.
$$

We now estimate $u$. Fix $k \geq 2$. We have
$$u(q) \leq -\frac{q}{1-q} - \sum_{n=1}^{k-1} \frac{8nq^{4n}}{1-q^{4n}} + \sum_{n=1}^{k-1}\frac{4nq^{2n}}{1+q^{2n}} -\sum_{n=k}^{\infty} 8nq^{4n} + \sum_{n=k}^{\infty} 4nq^{2n},$$
by splitting the sum between $n \leq k-1$ and $n \geq k$, and using the inequalities
$$
\frac{8nq^{4n}}{1-q^{4n}} > 8nq^{4n}
$$
and
$$
\frac{4nq^{2n}}{1+q^{2n}} < 4nq^{2n}.
$$
Evaluating the two infinite series, we obtain the following upper bound for $u(q)$:
$$
-\frac{q}{1-q} - \sum_{n=1}^{k-1} \frac{8nq^{4n}}{1-q^{4n}} + \sum_{n=1}^{k-1}\frac{4nq^{2n}}{1+q^{2n}} + \frac{8(k-1)q^{4(k+1)}-8kq^{4k}}{(1-q^4)^2} + \frac{4kq^{2k}-4(k-1)q^{2(k+1)}}{(1-q^2)^2}.
$$
Using \textsc{maple}, we can substitute different values of $k$ and solve where the resulting expression is negative.
With $k=10$, we obtain that the above expression is negative for $q \leq 0.81121$.
In particular, $f$ is decreasing in the interval $(0,0.8]$.
\bigskip

\textbf{Case 2: $q \in (0.8,1)$}

In this case, we shall make another change of variable, using the modularity of the theta-functions.
The Jacobi modular identity for theta-functions implies that
$$
\vartheta_2(e^{-\pi/x}) = \sqrt{x} \vartheta_4(e^{-\pi x}).
$$
Making the change of variable $q=e^{-\pi/x}$, we get
$$
f(q) = \frac{1}{4} (q^{-1/2}-q^{1/2}) \vartheta_2(q)^2
= \frac{1}{2}x \Bigl(\sinh{\frac{\pi}{2x}}\Bigr) \vartheta_4(e^{-\pi x})^2.
$$
Note that $x$ increases from $0$ to $\infty$ as $q$ increases from $0$ to $1$.
Furthermore, if $0.8<q<1$, then $x>\pi/(\log{5/4}) \approx 14.0788$.
It thus suffices to prove that $f(x)$ is decreasing for $x>14$.
Write $f(x)=g(x)h(x)$, where
$$
g(x):=\frac{1}{2}x \sinh{\frac{\pi}{2x}}
$$
and
$$
h(x):=\vartheta_4(e^{-\pi x})^2.
$$
It is easy to prove that $g$ is decreasing on $(0,\infty)$. Indeed, $g'(x)\leq0$ is equivalent to
$$
\tanh{\frac{\pi}{2x}} \leq \frac{\pi}{2x},
$$
which is true since $\tanh(\theta) \leq \theta$ for $\theta \geq 0$.

Now, we shall use the fact that $h(x) \approx 1$ for $x>14$ to deduce that in this case, the behavior of $f$ and $g$ are nearly the same. We have to do some numerical error analysis:

First, we need to estimate how close $h(x)$ is to $1$ when $x>14$. Note that
\begin{align*}
\vartheta_4(q) &= 1-2q+2q^4-2q^9+ \dots\\
&\geq 1-2q-2q^2-2q^3-\dots\\
&= 1- \frac{2q}{1-q}.
\end{align*}
Hence, for small $q$,
$$
\vartheta_4(q)^2 \geq 1+4 \frac{q^2}{(1-q)^2} -  4\frac{q}{1-q} \geq 1-4 \frac{q}{1-q}.
$$
Thus, with $q=e^{-\pi x}$,
$$
1-\vartheta_4(q)^2 = 1- \vartheta_4(e^{-\pi x})^2 \leq 4 \frac{e^{-\pi x}}{1-e^{-\pi x}} \leq 8 e^{-\pi x},
$$
since $1-e^{-\pi x} \geq 1/2$ for $x>14$.
Define $\xi(x):=1- \vartheta_4(e^{-\pi x})^2$, so that
\begin{equation}
\label{eq4}
0<\xi(x)\leq 8 e^{-\pi x},
\end{equation} for $x>14$.

We shall also need an estimate for the derivative of $g$:
$$
-g'(x) = -\frac{1}{2} \Bigl(\sinh{\frac{\pi}{2x}} - \frac{\pi}{2x}\cosh{\frac{\pi}{2x}} \Bigr)
= \frac{1}{2}\cosh{\frac{\pi}{2x}} \Bigl( -\tanh{\frac{\pi}{2x}} + \frac{\pi}{2x} \Bigr).
$$
Now, since $x > 14$, we have
$$
\frac{1}{2}\cosh{\frac{\pi}{2x}} \leq \frac{1}{2}\cosh{\frac{\pi}{28}} \approx 0.50315 <1.
$$
Also, for $\theta \geq 0$, $\theta - \tanh{\theta} \leq \theta^3/3$.
Indeed, both functions are $0$ at $0$ and if we compare the derivatives, we get
$$
\tanh^2{\theta}\leq \theta^2,
$$
which holds for every $\theta \geq 0$.
We thus obtain the following estimate for the derivative of $g$:
\begin{equation}
\label{eq5}
-g'(x) \leq \frac{1}{3} \Bigl(\frac{\pi}{2x} \Bigr)^3 = \frac{\pi^3}{24x^3} \leq \frac{\pi^3}{24(14^3)} \leq 1.
\end{equation}

Now we estimate $g(x)$. This is easy; $g$ is decreasing, so for $x > 14$, we have
\begin{equation}
\label{eq6}
g(x) \leq g(14) = \frac{1}{2} (14) \sinh{\frac{\pi}{28}} \approx 0.78705<1.
\end{equation}

We also have the easy estimate
\begin{equation}
\label{eq7}
\vartheta_4(e^{-\pi x}) \leq 1.
\end{equation}

Finally, we estimate the derivative of $\vartheta_4$:
\begin{align*}
\vartheta_4'(q) &= 2(-1+4q^3-9q^8+16q^{15}-\dots) \\
&\geq 2(-1-\sum_{n=1}^{\infty}(n+1)q^n)\\
&= -2 -4 \frac{q}{(1-q)^2} +2 \frac{q^2}{(1-q)^2}\\
& \geq -2 -4\frac{q}{(1-q)^2}.
\end{align*}
Hence,
\begin{equation}
\label{eq8}
-\vartheta_4'(e^{-\pi x}) \leq 2 +4 \frac{e^{-\pi x}}{(1-e^{-\pi x})^2} \leq 2+16e^{-\pi x} \leq 3,
\end{equation}
again for $x>14$.

We now have everything we need to estimate $f'(x)$:
\begin{align*}
f'(x) &= g'(x)h(x) + g(x)h'(x)\\
&= g'(x) - g'(x)\xi(x) - 2 \pi g(x) \vartheta_4(e^{-\pi x}) \vartheta_4'(e^{-\pi x}) e^{-\pi x}.
\end{align*}
Hence, by equations (\ref{eq4}) to (\ref{eq8}):
\begin{align*}
f'(x) &\leq g'(x) + 8 e^{-\pi x} + 2\pi(1)(1)(3)e^{-\pi x} \\
&= g'(x)+8 e^{-\pi x} + 6\pi e^{-\pi x}\\
&< g'(x)+27 e^{-\pi x},
\end{align*}
and this holds for all $x>14$.

To complete the proof, all that remains is to prove that this last expression is negative for $x>14$.
We proceed as follows:
\begin{align*}
g'(x)+27e^{-\pi x}
&= \frac{1}{2} \cosh{\frac{\pi}{2x}} \left( \tanh{ \frac{\pi}{2x}}-\frac{\pi}{2x} \right)+27e^{-\pi x}\\
&\leq  \frac{1}{2}\left( \tanh{ \frac{\pi}{2x}}-\frac{\pi}{2x} \right)+27e^{-\pi x}.
\end{align*}
However, we have the inequality
\begin{equation}
\label{eqtanh}
\tanh{\theta} \leq \theta - \frac{\theta^3}{3} + \theta^5 \qquad (\theta \geq 0),
\end{equation}
which follows from Taylor's theorem. Indeed, first note that
$$
\frac{d^5(\tanh{y})}{dy^5}
= -88 \operatorname{sech}^4y \tanh^2{y}+16\operatorname{sech}^6y+16\tanh^4{y}\operatorname{sech}^2y.
$$
Now, since $\tanh{y} \leq 1$ and $\operatorname{sech}y \leq 1$,
the remainder in Taylor's theorem for $\tanh{\theta}$ is less than
$(32\theta^5)/5! < \theta^5$, and inequality (\ref{eqtanh}) follows.
With $\theta=\pi/(2x)$, we obtain
\begin{align*}
g'(x)+27e^{-\pi x}
&\leq \frac{1}{2} \left(\frac{\pi}{2x} - \frac{1}{3}\left(\frac{\pi}{2x}\right)^3 +  \left(\frac{\pi}{2x}\right)^5-\frac{\pi}{2x}\right)+27e^{-\pi x}\\
&= 27e^{-\pi x} - \frac{\pi^3}{48}\frac{1}{x^3}+\frac{\pi^5}{64}\frac{1}{x^5}.
\end{align*}
Clearly the last expression is negative for $x \geq 14$, and thus we are done.
\end{proof}

\acknowledgment{The authors thank Vladimir Andrievskii, Juan Arias de Reyna, Dmitry Khavinson, Tony O'Farrell and Nikos Stylianopoulos for helpful discussions.}

\bibliographystyle{amsplain}

\end{document}